\documentclass{statsoc}

\usepackage[a4paper]{geometry}
\usepackage{graphicx}
\usepackage[textwidth=8em,textsize=small]{todonotes}
\usepackage{amsmath}
\usepackage{arydshln}
\usepackage{float}
\usepackage{url}

\usepackage{enumerate,amsxtra,comment}
\usepackage{dsfont}

\numberwithin{equation}{section}
\newcommand{\sgn}{{\operatorname{sgn}}}

\newcommand{\R}{\mathds{R}}

\newcommand{\dx}{\mathrm{d}}
\newcommand{\N}{\mathds{N}}
\newcommand{\Z}{\mathds{Z}}
\newcommand{\ind}{\mathds{1}}

\newcommand{\sumtn}{\sum_{t=1}^n}

\newcommand{\fn}{\hat {f_n}}
\newcommand{\ckn}{\hat c_{k,n}}
\newcommand{\dlambda}{{\dx \lambda}/(2 \pi)}
\newcommand{\dlambdaB}{\dx \lambda}
\newcommand{\intpi}{\int_0^{2 \pi}}
\newcommand{\gammah}{\widehat {\widehat{\gamma}}_n}
\newcommand{\sigmah}{\hat \sigma_n}
\newcommand{\Xs}[1]{X^*_{#1}}
\newcommand{\Es}{E^*}
\newcommand{\citeasnoun}{\citet}

\newcommand{\Xn}{\bar X_n}

\newtheorem{lemma}{Lemma}[section]
\newtheorem{theorem}{Theorem}[section]

\newtheorem{corollary}{Corollary}[section]

\newtheorem{definition}[lemma]{Definition}

\def\eps{\varepsilon}

\newcommand{\var}{\mbox{Var}}
\newcommand{\cov}{\mbox{Cov}}
\newcommand{\cum}{\mbox{cum}}

\title[Spectral Density-Driven Bootstrap]{Estimated Wold Representation and Spectral Density-Driven Bootstrap for Time Series}
  
\author[Krampe, Kreiss and Paparodits]{Jonas Krampe}
\address{TU Braunschweig - Inst. f. Mathematische Stochastik,
	Braunschweig,
		Germany}
\email{j.krampe@tu-bs.de}
\author[Krampe, Kreiss and Paparodits]{Jens-Peter Kreiss}
\address{TU Braunschweig - Inst. f. Mathematische Stochastik,
	Braunschweig,
		Germany}
\email{j.kreiss@tu-bs.de}
\author[Krampe, Kreiss and Paparodits]{Efstathios Paparoditis}
\address{University of Cyprus - Department of Mathematics and Statistics ,
	Nicosia,
		Cyprus}
\email{stathisp@ucy.ac.cy}

\begin{document}
\begin{abstract}
The second-order dependence structure of purely nondeterministic stationary process is described by the coefficients of the famous Wold representation. These coefficients can be obtained by factorizing the spectral density of the process. This relation
together with some spectral density estimator is used in order to obtain consistent estimators of 
these coefficients.
A spectral density-driven bootstrap for time series is then developed which uses
the entire sequence of estimated MA coefficients together with appropriately
generated pseudo innovations in order to obtain a bootstrap pseudo time series. 
It is shown that if the
underlying process is linear and if the pseudo innovations are generated by means of an
i.i.d. wild bootstrap which mimics, to the necessary extent, the moment structure of the
true innovations, this bootstrap proposal asymptotically works for a wide range of statistics.
The relations of the proposed bootstrap procedure to some other bootstrap procedures, including the
autoregressive-sieve bootstrap, are discussed. It is shown that the latter is a special case of the
spectral density-driven bootstrap, if a parametric autoregressive spectral density estimator is
used. Simulations investigate the performance of the new bootstrap procedure in finite
sample situations. Furthermore, a real-life data example is presented.
\end{abstract}

\section{Introduction}
\label{se.Intro}
The spectral density (SD), if it exists, plays an important role as a 
quantity which completely describes the so-called second-order properties of stationary time series. A broad literature
exists on SD estimators, among them parametric (e.g. autoregressive) estimators,
nonparametric (e.g. lag window or smoothed periodogram) estimators or semiparametric estimators 
as a mixture of both. Time series analysts typically are rather skilled in estimating spectral densities
and they know, depending on the required application, the pros and cons of the various 
estimators. This paper intends to bring together several bootstrap procedures under the umbrella of
SD estimation. 

Recall that for a purely nondeterministic and stationary stochastic process 
$X = (X_t, t \in \Z)$ with SD $f$,
Szeg\"o's factorization expresses $f$ as a power series. The coefficients of this factorization, 
appropriately normalized, coincide with the coefficients of the well-known Wold representation of
$X$. Recursive formulas, which make use of the
Fourier coefficients of $\log(f)$ - the so-called cepstral-coefficients -,  to calculate the coefficients 
of the Wold representation of the
process have been developed; cf. \citet{pourahamdi1983}. Moreover, if $f$ is strictly positive then $X$ 
also obeys an autoregressive (AR) representation and similar recursive formulas to compute
the coefficients of this representation have also been derived; see again \citet{pourahamdi1983}. 
Using these recursions
we suggest a procedure to estimate the coefficients of both the moving average (MA) and the 
AR representations based on an estimator $\hat f_n$ of the SD $f$. 
In particular, we show that under certain
conditions on $f$ and on the used estimator $\hat f_n$, the sequence of coefficients 
of the Wold and of the AR representation of the process can be consistently 
estimated. Furthermore, under additional smoothness conditions the pointwise consistency 
of the estimators can be extended to uniform consistency for the 
entire sequence of coefficients. It should be noted that the factorization of the spectral
density has been considered in the literature also for implementing and investigating 
the so-called Wiener-Kolmogorov predictor in linear prediction (cf. \citet{Jones1964},  \citet{Bhansali1974,Bhansali1977} and \citet{pourahamdi1983}).

The availability of estimates of the MA coefficients of the Wold representation
enables the development of a general spectral density-driven bootstrap (SDDB) procedure
for time series. In particular, a pseudo time series  can be generated by using 
the estimated sequence of MA coefficients and an appropriately chosen sequence
of pseudo innovations. The resulting bootstrap procedure is then fully determined
by the particularly chosen SD estimator $\hat f_n$ and the stochastic properties of
the generated pseudo innovations. The estimated Wold representation used should mainly be regarded
as a means to an end  to generate a pseudo time series which exactly has  the chosen SD estimator as its spectrum. %{\color{purple} The estimated Wold representation should mainly be regarded as a vehicle used to generate pseudo time series which exactly has  the chosen SD estimator as its spectrum.}% exactly has the chosen SD estimator as its spectrum.   

For instance, choosing a parametric AR SD estimator, the coefficients
of the estimated Wold representation coincide with the coefficients of the inverted 
estimated AR polynomial and therefore, the AR model can just as well be used
to generate the bootstrap data. In other words, using a parametric AR SD estimator
will lead to the well-known AR-sieve bootstrap for time series (cf. \cite{kreiss1992bootstrap}, 
\cite{buhlmann1997sieve} and \cite{kreiss2011}). However, a parametric
AR SD estimator often is not the first choice. Let us consider a nonparametric competitor, for instance, 
a lag window estimator of $f$ with truncation lag $M_n$. As we will see, this will lead us essentially 
to a MA process of finite order $M_n$ which can be used to generate the pseudo time series.  
Therefore, the SDDB proposed in this paper, is a general notion of 
bootstrap for time series which allows for a variety of possibilities to generate the 
pseudo time series. These possibilities are determined by the particular SD estimator
$\hat f_n$ used to obtain the
estimates of the coefficients of the Wold representation. 
Notice that although the SDDB generates bootstrap pseudo time series in the time domain, the
second-order dependence structure of the underlying process is entirely mimicked in the
frequency domain by means of the selected SD estimator used. Thus, various
well-known and flexible methods for SD estimation can be used in our bootstrap
method. As a consequence, we formulate the assumptions needed for our theoretical developments, in terms of the
SD and its estimator, only. This allows us to restrict the class of admissible SD 
estimators as little as possible. 

Fed by independent and identically distributed (i.i.d.) pseudo innovations the proposed SDDB generates pseudo time series stemming from a linear process. %{\color{green}, i.e. a MA process of possibly infinite order with i.i.d. innovations.}
For such a choice of pseudo innovations we compare our bootstrap proposal
to some other linear bootstrap procedures, like the AR-sieve bootstrap (cf. \cite{kreiss1992bootstrap} and \cite{buhlmann1997sieve}) and the linear process bootstrap, cf. \cite{mcmurry2010banded}. 
As already indicated, it is
shown that the AR-sieve bootstrap is a special case of the SDDB which is obtained if a parametric AR SD estimator 
$\hat f_n$ is used. {Furthermore, we show that
the linear process bootstrap essentially generates pseudo observations by factorizing banded autocovariance matrices. This technique is related to the factorization of spectral densities which is used in this paper. However, in finite samples the two approaches differ from each other.}

It is worth mentioning that pseudo innovations generated in a different way than i.i.d. could also be
used in the proposed SDDB procedure. For instance,
pseudo innovations generated by means of a block bootstrap applied to appropriately
defined residuals may be used. Although such a proposal would most likely extend
the range of validity of the SDDB to
nonlinear time series, %as is the case for the different block bootstrap methods, see \citet{lahiri2003resampling},
 we do not consider such an approach in this paper, i.e., we restrict ourselves to
the linear process set-up. We show that %if the underlying process X is linear and 
if the pseudo innovations are generated by means of an i.i.d. wild bootstrap that appropriately
mimics the first, the second, and the fourth moment structure of the true innovations,
then the proposed SDDB is asymptotically valid for a wide
range of statistics commonly used in time series analysis. Besides the sample
mean, statistics described by the so-called class of generalized autocovariances are also considered; see Section \ref{sect3} for 
details. 
%Notice that this class includes sample autocovariances, sample autocorrelations 
%and lag window SD estimators as special cases. 
We demonstrate by means of simulations that our asymptotic findings
coincide with a good finite sample behavior of the proposed bootstrap procedure.
Furthermore, the performance of the new bootstrap method is compared with that of the asymptotic normal approximations
and of some other bootstrap competitors, like the linear process, the AR-sieve
and the tapered block bootstrap. {An \emph{R}-code to generate pseudo time series with the SDDB is available at }\url{www.tu-bs.de/Medien-DB/stochastik/code-snippet_sddb.txt.}

The paper is organized as follows: Section \ref{sect2} briefly describes the Wold and the AR
representation of a stationary time series and discusses the method used
to estimate the entire sequence of coefficients in both representations. Local
and global consistency properties of the estimators are established. Section \ref{sect3} introduces
the SDDB procedure for time series and establishes, for linear
processes and for relevant classes of statistics, its asymptotic validity. A comparison
with the AR-sieve and with the linear process bootstrap is also given in this
section. Section \ref{sect4} presents some numerical simulations investigating the 
finite sample behavior of the proposed bootstrap method and compares its performance with that
of other bootstrap methods and of asymptotic normal approximations. A real-life data example
demonstrates the applicability of the suggested bootstrap procedure. Finally,
auxiliary results as well as proofs of the main results are deferred to Section \ref{sec:proofs}.

\section{Estimated Wold Representation}
\label{sect2}

\subsection{Moving Average and Autoregressive Representation}

Stationary processes are commonly  classified using the concept of linear prediction; see  for example \citeasnoun[Section 5.7]{BrockwellDavis1991} or  \citeasnoun[Section 5.5]{pourahmadi2001foundations}. %The three classes are defined depending on how well one can predict a stationary process.
To elaborate, let  $X=\{X_t, t \in \Z\}$ be a stationary stochastic process and define by  $\mathcal{M}_n(X)=\overline{\text{span}\{X_t, -\infty < t \leq n\}}$ and $\mathcal{M}_{-\infty}(X) =\bigcap_{n = -\infty}^\infty \mathcal{M}_n$  the closed linear subspaces of the Hilbert space $\mathcal{M}(X)=\overline{\text{span}\{X_t, t\in \Z\}}$. {Note that an overlined set denotes its closure.} Let $P_{\mathcal{M}_n(X)} X_{n+1}$ be the projection of $X_{n+1}$ onto $\mathcal{M}_n(X)$ and  define  by   $\sigma^2= E | X_{n+1} - P_{\mathcal{M}_n(X)} X_{n+1}|^2$ the mean square error of the best (in the mean square sense)  one-step,  linear predictor. The process $X$ is called  deterministic if and only if $X_t \in \mathcal{M}_{-\infty}(X)$ or equivalently if and only if $\sigma^2=0$. It is called nondeterministic if $X_{n+1} \not \in \mathcal{M}_n(X)$ and consequently $\sigma^2>0$. Furthermore, it is called  purely-nondeterministic if it is nondeterministic and $\mathcal{M}_{-\infty}(X)=\{0\}$. 

If the process $X$ possesses  a  SD $f$, which is the case  if  $ \sum_{h\in \Z}|\gamma_h| < \infty$, with  $\gamma_h=Cov(X_t,X_{t+h}) $,  then  it holds  true that $ X$  is nondeterministic if and only if 
\begin{align} 
\int_{(-\pi,\pi]} \log f (\lambda) d  \lambda > -\infty, \label{PND}
\end{align}
see  \citeasnoun[Theorem VII]{pourahmadi2001foundations}.

% of $(X_t)_{t\in \Z}$ and is by definition a white noise. 
%From a statistical point of view, nondeterministic processes are more interesting than deterministic. % processes.
Wold's decomposition, see  \citeasnoun[Theorem 5.11]{pourahmadi2001foundations},  guarantees   that any nondeterministic process can be divided into a deterministic and a  purely-nondeterministic part. Furthermore,  the  purely-nondeterministic part of the process has a unique one-sided moving average (MA) representation given by 
%. Consequently, let $X$ be PND with the innovation process $(\eps_t)_{t\in \Z}$, then we have
\begin{align} \label{eq.WoldRepr}
X_t=\sum_{k=0}^\infty c_k \eps_{t-k}, \ \ t \in \Z,
\end{align}
where $ \sum_{k}|c_k|^2 <\infty$ and  $\{\eps_t, t \in \Z\}$ is a white noise process   defined   by  $\eps_{n+1}=X_{n+1} - P_{\mathcal{M}_n(X)} X_{n+1}$, $ n \in \Z$, called the innovation process. Here, white noise refers to an uncorrelated time series.
Notice that even  if $X$ is a linear process driven by i.i.d. innovations, the  white noise process appearing in the corresponding  one-sided MA representation (\ref{eq.WoldRepr}) might not be  i.i.d.. To give an example consider the  linear, first order MA process,  $ X_t=e_t + \theta e_{t-1}$ where $ \{e_t, t\in \Z\}$ is an i.i.d. process and $ \theta >1$.  The Wold representation of this process is given by $X_t=\eps_t + \theta^{-1} \eps_{t-1}$ where $\eps_t=e_t +(1-\theta^2) \sum_{j=1}^\infty (-\theta)^{-j} e_{t-j}$  is white noise with variance $\theta^2$. Obviously, the innovations  $\eps_t$ are  not independent.

Another interesting one-sided representation of the process $X$ is the so-called AR representation which appears if 
 the SD $f$ is bounded away from zero, i.e., if $\inf_{\lambda\in [0,\pi]}f(\lambda) = C > 0$. In this case and instead 
 of using the full history of the innovation process $\{\eps_t, t \in \Z\}$,  
  the full 
 history of the process $X$ itself is   used to express  the  value  $ X_t$ at any time point $t$.  $X_t$ can then be written as   
\begin{align} \label{eq.ArRepr}
X_t=\sum_{k=1}^\infty b_k X_{t-k} + \eps_t,\quad  t \in \Z,
\end{align}
where $ \sum_{k}|b_k|^2 <\infty$ and $ \{\varepsilon_t \}$ is the same white noise innovation process as in (\ref{eq.WoldRepr}); see \citeasnoun[Section 6.2.1]{pourahmadi2001foundations}.  Expression (\ref{eq.ArRepr}) is called  the AR representation of   the process $X$ 
%Expression (\ref{eq.ArRepr}) driven by the white noise innovations  $\{\eps_t\}$ 
and should not be confused with that of a linear, infinite order AR process driven by i.i.d. innovations.  To demonstrate this,  consider again the  previous example of  the linear, noninvertible MA  process $ X_t =e_t + \theta e_{t-1}$ with $ \theta >1$. This process  has the AR representation $X_t=-\sum_{j=1}^\infty (-\theta)^{-j} X_{t-j} + \eps_t$ where $\{\eps_t, t\in \Z\}$ is the uncorrelated but not independent white noise processes appearing in the Wold representation of $X_t$. 

To derive recursive formulas for the coefficients in the MA representation $(\ref{eq.WoldRepr})$ and the AR representation $(\ref{eq.ArRepr})$, 
% In order to develop the  bootstrap procedure proposed in this paper 
%and which   uses   
%the  MA representation (\ref{eq.WoldRepr}) as a basic building block,
we start with some basic factorization  properties of the SD $f$. 
 Notice first that $f$ can be expressed as $f(\cdot)=(2\pi)^{-1}| V(\exp(-i\cdot))|^2$  for   a power series $V(z)=\sum_{k=0}^\infty v_k z^k$ and that such a factorization  exists  if and only if   condition (\ref{PND}) is fulfilled; see \citeasnoun{Szego1921}. The above factorization of the SD  is not unique. However,  if we restrict ourselves to  power series which have no roots inside the unit disk and appropriately normalize the coefficients, a unique representation occurs. The coefficients of this unique power series  coincide with  the coefficients $c_k$  of the Wold representation   (\ref{eq.WoldRepr}), if %additionally, the power series $V(z)$ is appropriately  normalized, i.e.,   if 
 $ \tilde V=V/v_0$ is used. We denote this unique and normalized power series  by  $C(z)=\sum_{k=0}^\infty c_k z^k$. %can be used for this purpose.  
Notice that  $(\ref{PND})$ ensures, that the Fourier coefficients of $\log f$ are well defined. Furthermore, since $C(z)$ has no zeros inside the open unit disc, $\log(C(z))$ is analytic inside the same  region  and we have for $|z| < 1$ that 
\begin{align}
 \sigma (2 \pi)^{-1/2}  \sum_{k=0}^\infty c_k z^k  =\exp \left( a_0/2 +\sum_{k=1}^\infty a_k z^k  \label{phase}\right),
\end{align}
where $a_k$ is the $k$-th Fourier coefficients of $\log f$, 
\begin{align}
a_k= \int_{-\pi}^\pi \log f (\lambda) \exp(- ik \lambda) d \lambda /(2\pi) \label{def_a_k}.
\end{align} 
Differentiation of equation $(\ref{phase})$ together with comparison of coefficients leads to a recursive formula to calculate the coefficients $\{c_k  \}$  of  this power series by using the Fourier coefficients of $\log f$, see \citeasnoun{pourahamdi1983,pourahamdi1984}. In particular, setting $c_0=1$,  the following recursive formula can be used to obtain the coefficients $\{c_k\}$, 
\begin{align}  \label{eq.ck}
c_{k+1}= \sum_{j=0}^k \left(1 -\frac{j}{k+1}\right)  a_{k+1-j} c_j, \quad  k=0,1,2,\dots. 
\end{align}
Furthermore,  $\sigma^2=2 \pi \exp(a_0)$.
If the process $X$ possesses also the AR representation (\ref{eq.ArRepr}), then the  coefficients $(b_k)_{k \in \N}$ of this representation can  be calculated  using the relation 
 $C(z)^{-1}=B(z)=\sum_{k=0}^\infty (-b_k) z^k$. Setting  $b_{0}=-1$ the corresponding recursive formula to obtain the $b_k$'s         is given by  
\begin{align} \label{eq.bk}
b_{k+1}= -\sum_{j=0}^k \left(1 -\frac{j}{k+1}\right)  a_{k+1-j} b_j, \quad  k=0,1,2,\dots.
\end{align}

{A proof of (\ref{phase}) can be found in the Supplementary Material, Lemma A.2%\ref{A-lemphase}
. As we see from the proof of (\ref{phase}), this approach cannot be transferred directly to the multivariate case. Matrix multiplication is not commutative and therefore the exponential laws do not apply for matrices. However, these properties are essential for the proof of (\ref{phase}). Moreover, there are examples where $(\ref{phase})$ is not valid in the multivariate case. Consequently, the recursive formulae $(\ref{eq.ck})$ and $(\ref{eq.bk})$ cannot be directly applied to multivariate time series. %An 
 
\subsection{Estimating the  Coefficients of the  Wold  Representation} \label{sec.2.2}
Our next goal is to estimate the coefficients $\{ c_k, k \in \N\}$ of the Wold representation  (\ref{eq.WoldRepr}). 
%, esIn order to estimate the Wold coefficients or the coefficients of the AR Representation, we can factorize an estimated SD $\fn$. Therefore, we have 
%\begin{align}
%\fn(\lambda)=\frac{\hat \sigma_n^2}{2\pi}\left|\sum_{k=0}^\infty \ckn \exp(-i k \lambda) \right|^2 = \frac{\hat \sigma_n^2}{2\pi} \left|1 - \sum_{k=1}^\infty \hat b_{k,n} \exp(-ik\lambda)\right|^{-2}
%\end{align}
The basic idea is to use an estimator $ \fn $ of the SD $f$ to 
get estimates of the Fourier coefficients of $ \log(f)$ and to plug in these estimates in the recursive formula (\ref{eq.ck}). Notice that estimates of the coefficients $\{b_k, k \in \N \}$ of the AR representation (\ref{eq.ArRepr}) can be obtained by using formula (\ref{eq.bk}) and the estimates of the $a_k$'s.
 
Let $\hat a_{k,n}=(2 \pi)^{-1}\int_{-\pi}^\pi \log ( \fn (\lambda)) \exp(-ik\lambda) d \lambda$ be the estimator of the $k$-th Fourier coefficient of $\log(f)$ and denote by $ \{\hat{c}_{k,n}$, $ k\in \N\}$,  the estimators of the coefficients of the Wold representation obtained using formula (\ref{eq.ck}), e.g. $\hat c_{0,n}=1,  \hat c_{k+1,n}= \linebreak\sum_{j=0}^k \left(1 -{j}/({k+1})\right) \hat a_{k+1-j,n} \hat c_{j,n}, \quad  k=0,1,2,\dots.$ Let $\{\hat{b}_{k,n}\}$ be the corresponding estimators of the coefficients of $\{b_k\}$ using formula (\ref{eq.bk}). 
Using the above recursive formula, it is theoretically possible to compute the infinite series of coefficients corresponding to $\fn$. However, in practice the computation of Fourier coefficients is usually approximated by a sum of finite frequencies. This limits the number of MA coefficients to compute and gives an approximation error. This error depends on the smoothness of $\log f$ and is usually negligible, see Supplementary Material A.1 for details. To give an example, consider Model II used in the simulation study; see Section 4. This model possesses the slowest decaying autocovariance of all three models considered. Nevertheless, using $1024$ instead of $8192$ Fourier frequencies to compute Wold's coefficients gives an overall squared error of less than $10^{-5}$. 
%{\color{red}To give an example, let $(X_1,\dots,X_{512})$ be observations of Model II used in the simulation study and consider a lag window SD estimator. Let  $\{\hat c_{k,n}^{\prime}, k\in \N\}$ be series obtained using $1024$ Fourier frequencies and $\{\hat c_{k,n}^{\prime\prime}, k\in \N\}$ be the series obtained using  $8192$ frequencies. Then, $\sum_{k=0}^{128} (c_{k,n}^\prime-c_{k,n}^{\prime\prime})^2 < 10^{-5}$.}

It is clear that the properties of the estimators $\hat{c}_{k,n}$ and $ \hat{b}_{k,n}$ depend heavily on the properties of the estimator $\fn$. 
To obtain consistency, the following condition suffices which essentially requires that  $\fn$ 
%satisfies (\ref{PND}) and that it 
is a uniformly  consistent estimator of $f$. For lag window estimators such a uniform consistency has been established by \citet[Lemma A.2]{Jentsch2015}, and for AR SD estimators by \citet[Theorem 3.2]{Buhlmann1995331}.
%%%
%%% Ausfuhren
%%%

\vspace*{0.2cm}
{\bf Assumption 1}\ The estimator $\fn$ satisfies  $ \int_{(-\pi,\pi]} \log(\fn (\lambda)) d  \lambda > -\infty$. Furthermore, 
\begin{align}
 \sup_{\lambda \in [0,\pi]}|\fn(\lambda) -f(\lambda)| \stackrel{P}{\rightarrow} 0, \text{ as } n \rightarrow \infty. \label{f_const}
 \end{align}
 
\vspace*{0.2cm}

Then, the following  result can be established. 
\begin{theorem} \label{consistence}
Suppose that  $f$  satisfies  (\ref{PND}) and that Assumptions 1 holds true. Then,  as $ n \rightarrow \infty$, 
a) $\sup_{k \in \N} |\hat a_{k,n}-a_k |  \overset{P}{\to}  0$, and for every fixed $k\in \N$,
b)  $\hat c_{k,n} \overset{P}{\to} c_k$, and c) \  $\hat b_{k,n} \overset{P}{\to} b_k$.
%\item \  $\hat b_{k,n} \overset{P}{\to} b_k$.
%\end{enumerate}
\end{theorem}
By the above theorem, for an $M$-dependent process, we have $\sum_{k=0}^M |c_k -\hat c_{k,n}|=o_P(1)$. Imposing more conditions  on   $f$ and its  estimator $\fn$, the consistency  properties  of the estimators $ \hat{a}_{k,n}$ and $ \hat{c}_{k,n}$  can be  refined  and  inequalities, similar to the well-known Baxter inequality for the AR-coefficients, \cite{Baxter1962},  can be established. Such  inequalities   are  useful since they control  the  overall  estimation error  that occurs when  the  estimated SD $\fn$     instead of the true SD  $f$ is used in order to obtain the estimates of  interest.

\vspace*{0.2cm}
{\bf Assumption 2} \  The estimator $\fn$ fulfills the following conditions.
\begin{enumerate}
\item[(i)] \ There exists constants $ 0 < C_1 < C_2 < \infty$ such that  $ C_1 \leq  \fn(\lambda) \leq C_2$ for all  $ \lambda \in [0,\pi]$ and     all $ n \in \N$. \label{bounded}
\item[(ii)]  \ The first derivative of $ \fn$ with respect to $\lambda$ exists,  is continuous and integrable. Furthermore, 
\begin{align} 
\sup_{\lambda \in [-\pi,\pi]} \left| \frac{\dx}{\dx \lambda}\fn(\lambda) - \frac{\dx}{\dx \lambda} f(\lambda) \right|\stackrel{P}{\rightarrow} 0,\text{ as } n \rightarrow \infty. \label{df_const}
\end{align}
\end{enumerate}
Condition (ii) can be verified for lag window estimators by using similar arguments as in  the proof of Lemma A.2 in \citet{Jentsch2015} under the same cumulant conditions and a slightly faster decay of the autocovariance function. %{However, the rate of convergence of the estimation error of  the derivative curve is slightly slower.}
For the AR SD estimators the same condition can be verified by using arguments similar to those used in the proof of Theorem 3.2 in \citet{Buhlmann1995331}.
Notice that boundedness of the SD is ensured by an absolute summable autocovariance function, which is a common assumption for bootstrap procedures for time series. Furthermore, the assumption regarding the existence of derivatives of the SD can be transferred to assumptions on the summability of the autocovariance function. However, since the bootstrap approach proposed in this paper is SD-driven, 
we prefer to formulate the conditions needed as assumptions for the SD of the underlying process. The following theorem summarizes the properties of the estimators  $ \{\hat{a}_{k,n},n\in \N\}$ and  $ \{\hat{c}_{k,n},n\in \N\}$.

%\vspace*{0.2cm}

\begin{theorem}\label{maseries}
Let the spectral density $f$  be strictly positive and bounded with  continuous and integrable first derivative.  Then,  as $ n \rightarrow \infty$, 
\begin{enumerate}
\item[(a)] \ If $\fn $ satisfies Assumption 1 and Assumption 2(i) then 
\begin{equation}
\sum_{k=-\infty}^\infty |\hat a_{k,n}- a_k |^2 = \int_0^{2 \pi} | \log f (\lambda) - \log \fn(\lambda)|^2  \dlambda  \stackrel{P}{\rightarrow} 0 \label{L2ak}
\end{equation}
and
\begin{equation}
\sum_{k=0}^\infty  |\hat c_{k,n}-c_k |^2 \stackrel{P}{\rightarrow} 0. \label{L2ck} 
\end{equation}
\item[(b)] \ If $\fn $ satisfies Assumption 1 and Assumption 2, then 
\begin{equation}
\sum_{k=-\infty}^\infty k^2 | \hat a_{k,n}-a_k|^2 \stackrel{P}{\rightarrow} 0 \label{L2akk} \text{ and } \sum_{k=1}^\infty |\hat a_{k,n}-a_k| \stackrel{P}{\rightarrow} 0.
\end{equation}
Furthermore, 
\begin{equation}
\sum_{k=0}^\infty  k^2|\hat c_{k,n}-  c_k |^2 \stackrel{P}{\rightarrow} 0   \text{ and } \sum_{k=0}^\infty  |  \hat c_{k,n}- c_k |   \stackrel{P}{\rightarrow} 0.
 \label{L1ck}
\end{equation}
\end{enumerate}
\end{theorem}

Relation (\ref{phase}) plays a key   role in  the proofs of  assertions (\ref{L2ck}) and (\ref{L1ck}).  Notice that since $C(z)^{-1}=B(z)$,   similar relations  for $\{b_k,k\in \N\}$ can be derived. Furthermore, the results of Theorem~\ref{maseries} can be straightforwardly  extended to the sequence of estimation errors $\{(\hat b_{k,n}-b_k)$, $k \in N\}$.

There are some alternative approaches   to   estimate  the coefficients $c_k$ and $b_k$   which have  been proposed in the literature.  In particular and for estimating the coefficients $c_k$, one option is   the innovation-algorithm which works by fitting MA(q) models where  the order $q$  increases to infinity as the sample size $n$ increases to infinity;  see  \cite[Proposition 5.2.2]{BrockwellDavis1991}. For 
estimating   the coefficients $b_k$,  commonly  an AR(p) model is fitted to the time series at hand by means of Yule-Walker estimators, where,   the  order $p$ 
is also allowed to increase to infinity with sample size;   see  \cite[Section 8.1]{BrockwellDavis1991}. Under certain conditions, both approaches  
are  consistent; see \cite[Theorem 7.14]{pourahmadi2001foundations}. However, the basic idea behind  these approaches differs from ours and so do the estimators obtained via SD factorization.  In the above mentioned approaches, the  estimated  autocovariance matrix is used to fit a finite MA  or a finite AR model. Consistency  of the corresponding estimators is  then obtained  by allowing the  order of the fitted model to increase to infinity at an appropriate rate as the sample size $n$ increases to infinity. These approaches face, therefore, two sources of errors. The first is  the  estimation error which is caused by   the fact that  estimated autocovariances are used instead of the true ones. The second is the  approximation error which is due to the fact that a finite order  model  is used to approximate  the underlying  infinite order structure. Although the  estimation error  cannot be avoided,  
  %our approach does not necessarily use a finite order model  since it is based on a (uniformly) 
  %consistent estimator $\fn$  of the SD $f$. So
   the approximation error caused by our estimation procedure is different. This error depends  on the quality of the   SD estimator $\fn$ used to approximate the true SD $f$, where $\fn$ is selected from a wide   range of possible estimators and not only  from   those obtained by  using  finite order AR or MA parametric models. The innovation-algorithm is similar to the factorization of autocovariance matrices which is used in the linear process bootstrap. In the Supplementary Material, see Section A.2, %\ref{A-LPB.supply}
a simple example is discussed to point out the differences between factorizing autocovariance matrices and spectral densities.
 \subsection{Spectral Density Estimators}
Since our estimation procedure relies on a SD  estimator $\fn$, we briefly discuss the variety of such estimators% that can be used
 and their impact  on the  estimators     $\{\hat{c}_{k,n}\}$ or $\{\hat b_{k,n}\}$ obtained.% obtained.

As already mentioned, spectral densities can be estimated  using a parametric approach, that is, by fitting a parametric model  to the time series at hand and  using 
 the SD of the fitted model as an estimator of  the SD of the process. Since AR models 
 are easy to fit,  they are commonly  used for such a purpose; see \citeasnoun{akaike1969fitting}, \citeasnoun{Shibata1981}. % and  \citeasnoun[Section 10.6]{BrockwellDavis1991}. 
In this context, parameter estimators,  like Yule-Walker estimators, are  popular because they 
 ensure  invertibility of the  corresponding estimated  AR-polynomial; see \citeasnoun[Section 8.1]{BrockwellDavis1991}. Now, if an AR SD estimator is used in the  spectral factorization procedure, then the estimated coefficients $\{\hat{c}_{k,n}\}$ obtained are identical to those appearing in the power series expansion of the inverted estimated AR polynomial. Furthermore, the corresponding  sequence of estimated coefficients  $\{\hat{b}_{k,n}\}$ is  finite and the $\hat{b}_{k,n}$'s, $k \in \{1,2\dots,p\},$ coincide with the estimated AR parameters. 

 Using nonparametric methods  like lag window or kernel smoothed periodogram  estimators is another popular approach to estimate the SD; cf. \citeasnoun[Section 10.4]{BrockwellDavis1991}.  Lag window estimators truncate the estimated autocovariances at a given lag controlled by  a truncation parameter. Such estimators of the  SD can be interpreted as obtained by (implicitly)  fitting a  
 finite order MA model to the time series at hand; see  also \citeasnoun[Prop. 3.2.1]{BrockwellDavis1991}. The sequence of estimated coefficients  $\{\hat{c}_{k,n}\}$ of  the Wold representation obtained  by using such a SD
  estimator  is  finite with  $\hat{c}_{k,n}=0$  for values of $k$ larger than the truncation parameter.  Due to the asymptotic equivalence between lag window and smoothed periodogram estimators, similar remarks can be made also 
  for SD estimators obtained by smoothing the periodogram. Furthermore, as mentioned in Section 2.2, lag window estimators as well as AR estimators satisfy Assumptions 1 and 2. %, see  \cite[Lemma A.2]{Jentsch2015} and \cite[Theorem 3.2]{Buhlmann1995331} respectively.

 A different  nonparametric approach to estimate the SD  is to truncate  the 
  Fourier series of $ \log(f)$    which presumes an  exponential model for the SD; see \citeasnoun{Bloomfield1973}. Such a model is given by 
$
f(\lambda)=(2\pi)^{-1}\sigma^2 \exp\{2 \sum_{j=1}^r \theta_j \cos(\lambda j)\}.
$
Unlike truncating the autocovariance function, non-negative definiteness of the SD $f$ is ensured for all possible values of the parameters $\theta_j, j=1,\dots,r$.
%which give for all parameter sets a plausible SD; unlike truncating the autocovariance function, where the truncated autocovariance function does not necessary fulfill the property of non-negative definiteness. 
As  \citeasnoun{Bloomfield1973} pointed out,  the autocovariance function of such an exponential model cannot, in general,  be described by a finite AR or a finite MA model. Thus,  using such an estimator of the SD in the  factorization algorithm, 
leads to an infinite sequence  of estimators $(\ckn)$ or  $(\hat b_{k,n})$ respectively. Notice that the Fourier coefficients of  $\log(f)$ are also known as the cepstral coefficients or vocariances and they have been widely used in the signal processing literature to estimate the SD; see \citeasnoun{Stoica2006}. %\citeasnoun{ARMA-Cepstrum} %However, \citeasnoun{Stoica2006} define the cepstral coefficients as the finite approximation over $N$ Fourier frequencies and the integral definition, as we have used in $(\ref{def_a_k})$, is called 'theoretical cepstrum'; see \citet[Eq. (8)]{Stoica2006}. The finite approximated cepstral coefficients cannot linked directly without error to Wold's coefficients. %The statistical properties of these SD estimators are not investigated broadly in the literature as it is the case for previous SD estimators. %Therefore, to the author's knowledge  it is not clear if these cepstrum based estimators fulfill the main conditions in assumption 1 and 2.

An interesting combination of  nonparametric and parametric approaches for SD estimation is offered 
  by the so-called pre-whitening approach; see \citeasnoun{blackman1958measurement}. The idea is to use a parametric model to filter the time series and then apply a  nonparametric estimator to the  time series of  residuals.  Using an AR-model for pre-whitening (filtering)  and  a lag window estimator for estimating the SD 
  of  the residuals, can be interpreted as (implicitly) fitting an ARMA-model to the time series at hand. The idea is that the parametric AR-model fit is able to represent the peaks of the SD  quite well while  the lag window estimator applied  to the 
  residuals can capture features  of the SD that are not covered  by the parametric fit. %Considering an AR-model fit of order $p$ and starting with a time series of length $n$, the lag window estimator is applied to the residuals of length $n-p$. Thus
Notice that for the pre-whitening approach consistency of the lag window estimator is obtained even in the case, where the parametric fit does not improve the estimation. %However, since only $n-p$ instead of $n$ observation are used the rate of converge is slightly slower. Consequently even for $n$ dependent $p$,  as long as $n-p(n) \to \infty$ as $n\to \infty$ the pre-whitening approach is consistent and fulfill. 
Using such a SD estimator for the factorization algorithm the coefficients $\{ \hat{c}_{k,n}\}$ and $\{\hat{b}_{k,n}\}$ obtained will be those of the infinite order MA representation and infinite order AR representation of the (implicitly) fitted ARMA model, respectively. However, to reduce numerical errors, the use of the ARMA representation is recommend, the MA coefficients are obtained by the factorization of the pre-whitened SD and the AR coefficients are those of the fitted AR-model.
\vspace{-0.3cm}
\section{Spectral Density-Driven Bootrstrap}
\label{sect3}

\subsection{The Spectral Density-Driven Bootrstrap Procedure}\label{sec.boot}
In the previous section we have dealt with the coefficients $\{c_k , k \in \N\}$ of the MA and $\{b_k , k \in \N\}$ of the AR representation of the process. For the coefficients in both representations, consistent estimators have been developed. Consequently, both representations can be used in principle to develop a bootstrap procedure to generate pseudo time series $ X_1^\ast, X_2^\ast, \ldots, X_n^\ast$. We focus in this work on the MA representation, since it exists for every SD. Clearly, such a bootstrap procedure will be determined by the SD estimator $\fn$ used to obtain the coefficients $ \{\hat{c}_{k,n}\}$ and by the generated series of pseudo innovations $\{\epsilon_t^\ast \}$ (cf. Step \ref{innovationgene} below). Thus, the tuning parameters of this bootstrap procedure coincide with those used for the SD estimation. Consequently, %for the choice of  tuning parameters
one can follow data-driven methods proposed in the literature to choose these parameters. %Let $\Xn$ be the sample mean.
Now, given an estimator  $\fn$  of the SD $f$,  the  SDDB algorithm consists of the following steps.

\begin{enumerate} [Step 1.]
\item Compute the Fourier coefficients of $\log (\fn)$ given by \linebreak$\hat a_{k,n}={1}/{(2\pi)} \int_0^{2\pi} \log( \fn(\lambda)) \exp (-i k \lambda) d \lambda$ for $k =1,2,\dots$.
\item Let  $\sigmah^2=2\pi \exp(\hat a_{0,n})$ and compute the 
coefficients $\hat c_{k,n}, k=1,2,\dots$ using the formula $\hat c_{k+1,n}= \sum_{j=0}^k \left(1 -{j}/({k+1})\right)  \hat a_{k+1-j,n} \hat c_{j,n},   k=0,1,2,\dots,$ and the starting value $\hat c_{0,n}=1$. 
\item Generate i.i.d. pseudo innovations  $\{\eps_t^*$, $ t\in \Z\}$ with mean zero and variance $\sigmah^2$. \label{innovationgene}
\item The pseudo time series  $X_1^*, X_2^*, \ldots, X_n^*$ is then obtained as 
$X_t^*=\sum_{j=0}^\infty \hat c_{j,n} \eps^*_{t-j} +\Xn,$  $t=1,2, \ldots, n,$ where $\Xn=n^{-1} \sum_{t=1}^n X_t$ is the sample mean.
\end{enumerate}
%The innovations in step \ref{innovationgene} can be generated for instance by a standard normal distribution or if the innovation shall have $\tilde \kappa_4$ as a fourth moment, one can use the following distribution
%\begin{align*}
%P\left(\eps_t^*=\sqrt{\tilde\kappa_4}\right)=P\left(\eps_t^*=-\sqrt{\tilde\kappa_4}\right)=\fraco{2\tilde\kappa_4}, P\left(\eps_t^*=0\right)=1-\fraco{\tilde\kappa_4}.
%\end{align*}

 It should be stressed that the above bootstrap algorithm with i.i.d. pseudo innovations represents a general procedure to generate a pseudo time series stemming from a linear process. %which is determined by the particular SD estimator $\fn$ used. 
Regarding the particular generation of the i.i.d. innovations in Step  \ref{innovationgene},  different possibilities can be considered depending on the stochastic properties  of the time series at hand which should be mimicked by the pseudo time series $ X_1^\ast, X_2^\ast, \ldots,, X_n^\ast $. In particular, suppose that $ X_1, X_2, \ldots, X_n$ stems from a linear process and that a statistic $T_n=T(X_1, X_2, \ldots, X_n)$  is considered, the distribution of which should be approximated by the bootstrap. 
We then propose to generate the i.i.d. innovations in a way which 
%In this case the pseudo innovations can be obtained  as an i.i.d. sequence generated to  
asymptotically matches the first, the second and the fourth moment structure of the true innovations $ \varepsilon_t$. Matching also the fourth moment structure of $ \varepsilon_t$ turns out to be  important for some statistics $T_n$; we  refer to  Section~\ref{sec.bootval} for examples.

One possibility to achieve this requirement is, to generate the $\eps_t^*$'s as i.i.d. random variables  with  the  following discrete distribution: $P\left(\eps_t^*=\sigmah\sqrt{\tilde\kappa_4}\right)=P\left(\eps_t^*=-\sigmah\sqrt{\tilde\kappa_4}\right)=1/({2\tilde\kappa_4})$ and  $P\left(\eps_t^*=0\right)=1-1/{\tilde\kappa_4}$. Here $\tilde \kappa_4=\tilde \kappa_{4,n}/\sigmah^4>0$ and  $\tilde \kappa_{4,n}$ denotes a consistent  estimator of the fourth moment $E (\eps_1^4)$ of the innovations $\varepsilon_t$. Consistent, nonparametric estimators of $ \kappa_4$ have been proposed in \citeasnoun{JPK2012} and \citeasnoun{Fragkeskou2015}.

In the above bootstrap  algorithm, the pseudo time series $X_1^*, X_2^*, \ldots, X_n^*$ 
is generated using the estimated  coefficients of the moving  average representation. Modifying the algorithm  appropriately, the pseudo time series can be also  generated using the estimated AR representation of the process.  For this, 
we set  $\sigmah^2=2\pi \exp(\hat a_{0,n})$ 
and  calculate the coefficients $\hat b_{k,n}$, $ k=0,1,2,\dots$  using the recursive formula starting with $\hat b_{0,n}=-1$
and $\hat b_{k+1,n}= -\sum_{j=0}^k \left(1 -{j}/({k+1})\right)  \hat a_{k+1-j,n} \hat b_{j,n},$
for $   k=0,1,2,\dots $. Using these estimates of  the  coefficients of the AR representation, the pseudo time series is then obtained  as  $X_t^*= \sum_{j=1}^\infty \hat b_{j,n}  (X_{t-j}^*-\Xn) +  \eps_t^* +\Xn$.

We stress here the fact that the SDDB should not be considered as an MA-sieve bootstrap procedure, where the order of the MA model is allowed to increase to infinity as the sample size increases to infinity. The SDDB procedure is rather governed by the SD estimator $\fn$ used, which appropriately describes the entire autocovariance structure of the underlying process.  The MA representation in this bootstrap procedure is solely used as a device to generate a time series with a second-order structure characterized by the SD estimator $\fn$ used. Notice however, that some SD estimators can implicitly lead to an MA-sieve type bootstrap.

\subsection{Comparison with other Linear Bootstrap Procedures} \label{sec.comp}
The idea of the AR-sieve bootstrap is to fit a $p$-th order  AR model to the time series at hand 
and to use the estimated  model structure together with i.i.d. pseudo innovations  generated according to the empirical distribution function of the centered residuals. In order to fully cover the second-order dependence structure of the underlying process $X$, the order $p$ of the fitted AR-model is allowed to increase  to infinity (at an appropriate rate) as the sample size increases to infinity; see \citeasnoun{kreiss1992bootstrap}, \citeasnoun{paparoditis1991order}, 
and \citeasnoun{buhlmann1997sieve}.
%The AR-sieve bootstrap has been introduced by  \citeasnoun{kreiss1988asymptotic};  see also \citeasnoun{paparoditis1991order}  \citeasnoun{kreiss1992bootstrap} and \citeasnoun{buhlmann1997sieve}. 
The range of validity of this bootstrap procedure has been investigated in \citeasnoun{kreiss2011}.   %Consequently, the second-order structure of the fitted AR-model defines the second-order structure of the pseudo innovations. 
%That is why 
As already mentioned, the AR-sieve bootstrap is  a special case of the SDDB described in Section \ref{sec.boot} when $\fn$ is chosen to be 
a parametric  AR$(p)$ SD  estimator and the innovations $\{ \eps_t^*\}$ are generated  through i.i.d. resampling from the centered  residuals of the AR fit.  Using the estimated AR-parametric SD, the factorization algorithm leads  to a sequence $\{\hat{c}_{k,n} \}$ of estimated MA coefficients that correspond to the MA$(\infty)$ representation obtained by inverting the estimated AR polynomial. %Thus, both the AR and the MA representation leads exactly to the same pseudo time series.
 %Using the $p$th order estimated autoregression or the MA($\infty$) representation,  both approaches lead  to 
%the same model. 
%Furthermore such a MA bootstrap which also used the estimated innovations leads to exactly the same bootstrap procedure as %the AR-Sieve.
However, and as already mentioned, the SDDB is a much more general  procedure since it 
is not restricted  to describing the dependence structure of the time series at hand by means of a finite order parametric AR model.  
{Notice that both bootstrap approaches work under similar conditions, see Assumptions $1$ and $2$. However, if a lag window SD estimator is used, there are situations where the SDDB is valid, whereas validity of the AR-sieve is not clear; see Section \ref{sec.bootval} for details.}

The linear process bootstrap, established by \citeasnoun{mcmurry2010banded} is also related to the SDDB. It uses the factorization  of banded autocovariance matrices instead of the SD itself to generate the pseudo observations. A factorization of autocovariance matrices is similar to the innovation algorithm, see \citet[Proposition 5.2.2]{BrockwellDavis1991}. As pointed out at the end of Section \ref{sec.2.2} this leads in finite sample situations to different results. % than a factorization of SD.
 Furthermore, the linear process bootstrap aims to generate a data vector with a given covariance structure, while the SDDB generates a stationary time series. A  more detailed discussion can be found in the Supplementary Material.
%Consequently, even if on first sight the linear process bootstrap (LPB) looks rather similar to the MA Sieve, there are some structural differences in the details. 

\subsection{Bootstrap Validity} \label{sec.bootval}
%Let $(X_t)_{t\in \Z}$ be a stationary, purely nondeterministic process with SD $f$, autocovariance function $\gamma$ and let $X_1,\dots, X_n$ be an observation segment of this process. For simplicity, assume that $\mu=E(X_t)=0$. Furthermore 
%let $\hat f_n$ be an estimator for the SD, which fulfill the Aassumptions 2.
 
In this section we prove validity  of the proposed SDDB procedure for the sample mean and under quite general dependence assumptions on the underlying process which go far beyond linearity. Furthermore, we show that if  the underlying process is linear,  the same bootstrap procedure driven by i.i.d. pseudo innovations is valid for the   class of so-called generalized autocovariance statistics. We first focus on this  general class  of  statistics.
% which appears to be more involved than that of the mean. 

%\subsection{Validity for generalized autocovariance statistic}
\begin{definition} \label{statistic}
Let $\{d_p(n), {n \in \Z}\}$ be a sequence of real numbers such that \linebreak $\sum_{h \in \Z} | d_p(h)|<\infty$, where $p \in \{1,2,\dots,P\}$. Let further $g: \R^P \to \R$ be a differentiable function. Then, the {\it generalized autocovariance statistic} is defined as
%\begin{align}
\begin{equation}
\hat{T}_n=g( \hat{T}_{n,1},\dots,\hat{T}_{n,P}), \text{ where for  }  p \in \{1,\dots,P\}, \label{eq.stat.Tn}
\end{equation}
%\end
$ \,\hat{T}_{n,p}= 1/n \sumtn \sum_{h=1-t}^{n-t} d_p(h) (X_t - \Xn) (X_{t+h}-\Xn)$ and  $\Xn = 1/n \sum_{t=1}^n X_t.$  
\end{definition}

The above class of statistics contains, among others, sample autocovariances, sample autocorrelations and lag window SD estimators, cf. the Supplementary Material for details. %To elaborate,  let $h\in \{0,\dots,n-1\}$ and set $d_1(h)=1$ and $d_1(x)=0$ for  $x \not=h$. We then have that $\hat T_{n,1}={1}/{n} \sumtn \sum_{s=1}^n \ind_{\{t-s=h\}} (X_t - \Xn)( X_s-\Xn) =\hat \gamma_n(h).$
%Similarly for $d_2(0)=1$ and $d_2(x)=0$ for  $x\not =0$, we get that $\hat T_{n,2}=\hat \gamma_n(0)$. Furthermore,   the sample autocorrelation at lag $h$ is obtained  by choosing $g(x,y)=x/y$. Lag window SD estimators are also 
%included in the above class. For this,  one chooses 
%$d(h)= \ind_{\{h \leq M\}} 1/({2 \pi}) K\left({h}/{M}\right) ( \ind_{\{h=0\}}+ \ind_{\{h>0\}}(\exp(-ih \omega)+\exp(i h \omega))),$
%where $K$ is some appropriate smoothing kernel. 

\vspace*{0.2cm}

{\bf Assumption 3:} \   $\{X_t,{t\in \Z}\} $ is a linear process $X_t=\sum_{j \in \Z} \varphi_j \eps_{t-j}+\mu, \mu \in \R$ with i.i.d. innovations $\{\eps_t,  t \in \Z\}$, where $E \eps_t=0$,  $E \eps_t^2=\sigma_\eps^2$, $ E \eps_t^4= \kappa_4$ and $E \eps_t^8<\infty$. We write for short $ \eps_t \sim IID(0,\sigma_\eps^2,\kappa_4)$. The coefficients in the MA representation fulfill the summability condition $\sum_{j \in \Z} |j \varphi_j|< \infty$.

\vspace*{0.2cm}

%Under the above assumptions, asymptotic normality of the statistic $\hat T_n$ can be established.
As the following theorem shows, the proposed SDDB procedure is valid for approximating the distribution of statistics belonging the class of generalized  autocovariances.  Here and in the sequel, for two random variables $X$, and $Y$,  $ d_2(X,Y)$ denotes Mallow's distance, i.e., $d_2(X,Y)=\{ \int_0^1 \left(F_X^{-1} (x)-F_Y^{-1}(x)\right)^2 dx\}^{1/2}$, where $F_X$ and $F_Y$ denote the cumulative distribution functions of $X$ and $Y$, respectively.%the asymptotic normality for the MA-Sieve-bootstrap estimation of the distribution of this class of statistics. %For simplicity, the theorem is stated to a single $\tilde T_{n,p}$. Since $g$ is a smooth function, the delta-method can be easily applied to show the asymptotic normality of $\hat T_n$.
\begin{theorem}\label{validity1}
Let 
$
\hat{T}_n^*=g( \hat{T}_{n,1}^*,\dots,\hat{T}_{n,P}^*),
$
where for $p \in \{1,\dots,P\}$,  
$\hat T_{n,p}^*= \linebreak {1}/{n} \sumtn \sum_{h=1-t}^{n-t} d_p(h) (X_t^* -\Xn^*) (X_{t+h}^*-\Xn^*),$
and $(d_p(h))_{h \in \Z}$ is a sequence of real numbers  as  in Definition \ref{statistic}. Furthermore,   $X_1^*, X_2^*, \ldots, X_n^*$   is a pseudo time series generated using the SDDB procedure with a pseudo innovation process  $ \{\eps^*_t, t \in \Z\}$   satisfying 
$ \eps^*_t \sim IID(0,\sigmah^2 , \hat \kappa_{4,n}) $ with   $\hat \kappa_{4,n}=E^*(\eps^*_t)^4$,    a consistent estimator of $\kappa_4$  which also fulfills   $\sup_{n\in \N} \hat \kappa_{4,n} \leq C$ for some constant $C<\infty$ which does not depend on $n$.  Finally, assume  that  the estimated Wold coefficients fulfill $\sum_{k \in \N} | c_k - \hat c_{k,n}| =o_P(1)$ and $\sum_{k \in \N} |k \hat c_{k,n}|\leq C$.
Then under Assumption 3 and as $n\to \infty$,
$
d_2 ( \sqrt{n} ( \hat T_n^* - E^* \hat T_n^*), \sqrt{n} ( \hat T_n - E \hat T_n)) \to 0, \ \mbox{in probability}.
$
%with $\sigma^2= \sum_{h_1 \in \Z} d(h_1) \sum_{h_2 \in \Z} d(h_2) \left( \gamma(h_1) \gamma(h_2) (k_4-3) + \sum_{k \in \Z} \gamma(k +h_1-h_2) \gamma(k) + \gamma(k+h_1)\gamma(k-h_2) \right)$.
\end{theorem}
The assumptions $\sup_{n\in \N} \hat \kappa_{4,n} \leq C < \infty$ and $\sum_{k \in \N} |k \hat c_{k,n}|\leq C$ are of rather technical nature  and can be satisfied  by using  appropriate estimators of $\hat{\kappa}_4$ and $\fn$. If the SD estimator $\fn$ fulfills $\sup_{\lambda \in (-\pi,\pi]} \frac{\dx^3}{\dx \lambda^3} \log \fn (\lambda) \leq C$ then the requirement $\sum_{k \in \N} |k \hat c_{k,n}|\leq C$ of the above theorem is satisfied. Notice that sufficiently smooth kernels guarantee the required differentiability of $\log \fn$. Furthermore, by using an appropriate truncation, boundedness of $\hat \kappa_{4,n}$ and $\fn$ can also be guaranteed.

In Section \ref{sect2} we gave conditions under which $\sum_{k \in \N} | c_k - \hat c_{k,n}| =o_P(1)$ holds, see Theorem \ref{consistence} and \ref{maseries}. Moreover, there are settings in which it is not clear whether the AR-sieve bootstrap is valid while the SDDB in connection with a lag window SD estimator can lead to a valid approximation. For instance, the SDDB remains valid for statistics $\hat{T}_n$ as in $(\ref{eq.stat.Tn})$ when the time series is generated by finite MA processes with unit roots, like for instance by the process $X_t=\eps_t - \eps_{t-1}$ or even by nonlinear continuous transformations of $M$-dependent stationary processes.

The following theorem establishes  validity of the SDDB for the case of the sample mean, which is not 
covered by the class of general covariance statistics $T_n$ given in  (\ref{statistic}). Notice, that for this case, it suffices that the pseudo innovations $ \{\varepsilon_t^\ast\}$   mimic asymptotically correct only the first and the second moment of the true innovations $\varepsilon_t$. Furthermore, no linearity assumptions of the underlying processes $X$ are needed. What is needed is that $\sqrt n (\bar X_n-\mu)$ converges to a normal distribution with variance $2\pi f(0)$, which, however, is fulfilled for a huge class of stationary processes. For instance, appropriate mixing or weak dependence conditions are sufficient for this statistic to satisfy the required asymptotic normality of $\sqrt n (\Xn-\mu)$. Furthermore, regarding the SDDB, the SD $f$ and its estimator $\fn$ need to fulfill less restrictive conditions. In particular, for a lag window SD estimator $\fn$, the assumptions $|\gamma(h) |\leq C /|h|^{2+\eps}$ and $\sup_t \sum_{t_1,t_2,t_3} |\cum(X_t,X_{t_1},X_{t_2},X_{t_3})|<\infty $, see \citet[Lemma A.2]{Jentsch2015}, suffice to ensure uniform consistency of $\fn$. %Under the same cumulant condition and the absolute summability of autocovariance function, consistency of block bootstrap approaches can be established for the sample mean; see \citet{kunsch1989}, \citet{politis1994stationary}, and \cite{nordman2009} for details.  %Thus, the SDDB is applicable in similar settings as the block-related bootstrap approaches.

\begin{theorem}\label{validity2}
Assume  that  $\{X_t : t \in \Z\}$ is a purely nondeterministic stationary process with mean $\mu$, SD $f$, and autocovariance $\gamma$ with $\sum_h |\gamma(h)| <\infty$ and assume that $\sqrt n (\Xn-\mu) \to \mathcal{N}(0,2\pi f(0))$, as $n\to \infty$. Denote by $\fn$ a uniformly consistent and bounded estimator of $f$ fulfilling Assumptions $1$ and $\sum_{k=0}^\infty |\ckn| <C,$ where $C$ does not depend on $n$. Assume that $X_1^*, X_2^*, \ldots, X_n^*$ is generated using the SDDB procedure with an i.i.d. innovation process  $ \{\eps^*_t, t \in \Z\}$,  where  $ E^*(\eps_t^*)=0$, $ E^*(\eps_t^*)^2=\sigmah^2$, and $E^*(\eps_t^*)^4<C<\infty$. Then, as $n\to \infty$,
$ d_2(\sqrt n  (\bar{X}_n^*- \Xn), \sqrt n  (\bar{X}_n-\mu)) \to 0$, in probability.
\end{theorem}

The assumption $\sum_{k=0}^\infty | \ckn| < C$ is satisfied if a strictly positive, differentiable, and bounded SD estimator $\fn$ is used.

Notice that validity of block bootstrap approaches is often established for so-called generalized mean statistics, see \citet[Example 2.2]{kunsch1989}. For a time series $X_1,\dots,X_n$, this class of statistics is given by
$T_n=h\left(1/(n-m+1) \sum_{t=1}^{n-m+1} Y_t\right)$, where $h : R^k \to R^s, s\leq k,$ and $Y_t=g(X_t,X_{t+1},\dots,X_{t+m-1}), t=1,\dots,n-m+1, g : \R^m \to R^k, k \leq m<n$ and $m$ is fixed. Let $\tilde n=n-m$. The validity of the SDDB for this class can be derived by applying the results of Theorem \ref{validity2}. The stated cumulant and autocovariance conditions have to be fulfilled by the process $\{Y_t, t \in \Z\}$.
\begin{corollary}\label{corollary0}
Let $Y=\{Y_t : t \in \Z\}$ fulfill the assumptions of Theorem \ref{validity2} and denote the mean by $\mu_Y= E Y_1$. Furthermore, assume that $h$ is differentiable at $\mu_Y$ and $Y_1^*,\dots,Y_{\tilde n}^*$ is generated using the SDDB procedure with an i.i.d. innovation process  $ \{\eps^*_t, t \in \Z\}$,  where  $ E^*(\eps_t^*)=0$, $ E^*(\eps_t^*)^2=\sigmah^2$, and $E^*(\eps_t^*)^4<C<\infty$. Then, as $\tilde n\to \infty$,
$ d_2(\sqrt {\tilde n}  (h(\bar{Y}_{\tilde n}^*)- h(E^* Y*)), \sqrt {\tilde n}  (h(\bar{Y}_{\tilde n})-h(\mu_Y)) \to 0$, in probability.
\end{corollary}

An improved finite sample performance of bootstrap approximations is often achieved by applying the bootstrap to studentized statistics, see for instance \citeasnoun[Chapter 6]{lahiri2003resampling}; \cite{gotze1996second,romano2006improved}. A studentized form is obtained by normalizing the statistic of interest with a consistent estimator of the asymptotic standard deviation. Since in Theorem \ref{validity2} the asymptotic variance is given by $2\pi f(0)$ and this quantity can be consistently estimated, we get $\sqrt n (\bar{X}_n- \mu)/(2\pi \tilde f_n(0))^{-1/2}$ as a studentized statistic where $\fn$ is a consistent estimator of $f$. A bootstrap approximation of this studentized statistic is then given by 
$\sqrt n (\bar{X}_n^*- \Xn)/(2\pi \tilde f_n^*(0))^{-1/2}$, where $\tilde f_n^*$ is the same SD estimator as $\tilde f_n$ obtained using the pseudo observations $X_{1}^*,\dots,X_n^*$. 

\begin{corollary}\label{corollary1}
Let $f(0)>0$ and $\tilde f_n(0), \tilde f_n^*(0)$ be consistent estimators of $f(0)$ which are bounded from below by $\delta>0$. Under the assumption of Theorem \ref{validity2} and if the SD estimator used for the SDDB is two times differentiable with a second derivative of bounded variation independent from $n$, then, as $n \to \infty$, $ d_2(\sqrt n  (\bar{X}_n^*- \Xn)  /(2\pi \tilde f_n^*(0))^{1/2}, \sqrt n  (\bar{X}_n-\mu)/(2\pi \tilde f_n(0))^{1/2}) \to 0$, in probability.
\end{corollary}

 The asymptotic variance of the generalized autocovariance statistic depends on the SD and it may also depend on the fourth moment $\kappa_4$ of the underlying innovations of the linear process. This fourth moment can be estimated consistently, by say $\hat \kappa_4$; see \citeasnoun{Fragkeskou2015}. Since the pseudo time series $\{X_t^*\}$ is driven by i.i.d. innovations, the fourth moment of $\{\eps_t^*\}$ can be estimated  using the same estimator as for $\kappa_4$. Consequently, an asymptotically valid approximation of the SDDB for studentized generalized autocovariance statistics can be established. This is done in the following corollary, where, and in order to simplify notation, only the case $P=1$ is considered. In this case the statistic of interest is given by $\hat{T}_{n}= 1/n \sumtn \sum_{h=1-t}^{n-t} d(h) (X_t - \Xn) (X_{t+h}-\Xn)$ and the asymptotic variance by $\tau^2=(\kappa_{4}/\sigma^4-3) (\int_{0}^{2\pi} f(\lambda) \sum_{h \in \Z} d(h) \exp(ih\lambda)\dlambdaB)^2 + 4 \pi \int_{0}^{2\pi} |f(\lambda)  \sum_{h \in \Z} d(h) \exp(ih\lambda)|^2 \dlambdaB$. 
\begin{corollary}\label{corollary2}
Let $\tau^2>\delta>0 $ and let  $\tilde f_n, \tilde f_n^*$ be consistent SD estimators which are bounded from below by $\delta>0$. Furthermore, let $\tilde \kappa_{4,n}, \tilde \kappa_{4,n}^*$ be consistent estimators of $\kappa_{4}$.  Under the assumptions of Theorem \ref{validity1} and if $E^* (\eps_1^*)^8<C$ independent from $n$ then, as $n \to \infty$,
$d_2 ( \sqrt{n} ( \hat T_n^* - E^* \hat T_n^*)/\tilde \tau_n^*, \sqrt{n} ( \hat T_n - E \hat T_n)/\tilde \tau_n)\to 0$, in probability.
\end{corollary}
The assumption $\tau^2>\delta>0$ ensures that $\hat T_n$ converges to a non-degenerate distribution. It is fulfilled if $\kappa_4/\sigma^4>\tilde \delta >1$ or if $f(\cdot) \sum_{h\in \Z} d(h) \exp(ih \cdot)$ is a non-constant function. The estimators $\tilde \tau_n^*$ and $\tilde \tau_n$ are estimators of $\tau$ based on $\tilde f_n^*$ and  $\tilde \kappa_{4,n}^*$ and $\tilde f_n$ and $\tilde \kappa_{4,n}$, respectively. 
\section{Numerical Examples}
\label{sect4}
\subsection{Simulations}
In this section we investigate by means of simulations the finite sample behavior  of   the SDDB and compare its performance with that of two other linear bootstrap methods, the AR-sieve bootstrap and the linear process bootstrap. We also compare all three linear bootstrap methods with the tapered block bootstrap, cf. \cite{paparoditis2001tapered}, and the moving block bootstrap, cf. \cite{kunsch1989}.  Two statistics $T_n$ are considered, the sample mean $ \overline{X}_n$ and the sample autocorrelation $ \widehat{\rho}(2)=\widehat{\gamma}(2)/\widehat{\gamma}(0)$. The  time series used  have been generated from the following three models:
\begin{itemize}
\item[]Model I: \  \  $ X_t = 0.9 X_{t-1} + \eps_t$,

\item[] Model II: \  
%\begin{align*}  
$X_t  = 1.34 X_{t-1} -1.88 X_{t-2} + 1.32 X_{t-3} -0.8 X_{t-4}  + \eps_t  + 0.71 \eps_{t-1}+ 0.25 \eps_{t-2}$
%\end{align*}
\item[]  Model III: \  $X_t=\eps_t + \sum_{k=1}^{10} \binom{n}{k} (-1)^k \eps_{t-k}$
%\item [] Model II: \ \ $ X_t= \eps_t -1.08 \eps_{t-1} + \eps_{t-2}$,
\end{itemize}
In all cases the innovation process $ \{\eps_t\}$ consists of i.i.d. random variables having a $t$-student distribution with 3 degrees of freedom and variance normalized to $1$.
 Model I is tailor made for the AR-sieve bootstrap. %a difficult to estimate AR(1) process with an AR  parameter close to the boundary of the unit interval. 
The SD in Model II has  
a difficult to estimate strong peak  around  frequency $\lambda=1.5$. Furthermore, this model possesses a slowly decaying autocovariance function which oscillates with two frequencies; one for the odd lags and one for the even lags. Model III is an MA process with a unit root; { the SD is zero at frequency zero. Consequently, the sample mean converges to a degenerated distribution making a studentization inappropriate.}
In order to investigate the finite sample 
performance of the different bootstrap methods, empirical coverage probabilities of two-sided confidence intervals obtained for the levels $\alpha=0.2, 0.1$ and $0.05$ are presented. The empirical coverage probabilities are based on $2,000$ realizations of each process and $B=1,000$ bootstrap repetitions. We present the results for the case  $n=128$, while results for the case $n=512$ are given in the Supplementary Material.

For the AR-sieve bootstrap, denoted by \emph{ARS}, the Akaike's information criterion (AIC) is used to select  the AR order $p$, cf. \cite{akaike1969fitting}. The \emph{SDDB} is applied using an AR-pre-whitening, nonparametric estimator of the SD, where the order of the  AR part has been selected  by the AIC and a smoothed periodogram is used with Gaussian kernel and of bandwidth selected by cross-validation; see \citeasnoun{beltrato1987determining}.
%and the nonparametric approach of \citeasnoun{paparoditis2012nonlinear} is used
% it is denoted by \emph{SDDB}. %The second implementation uses a lag window estimator with trapezoid kernel and the truncation rule of \citeasnoun{politis2003adaptive}; this is denoted by \emph{SDDB}. 
Furthermore, for this bootstrap procedure, i.i.d. Gaussian innovations are used. Furthermore, the linear process bootstrap, denoted by \emph{LPB},  has been implemented as in \citeasnoun{mcmurry2010banded}, and the tapered block bootstrap, denoted by \emph{TBB}, has been applied with a block length choice and a tapering window as in \citeasnoun{paparoditis2001tapered}. Due to the strong dependence of some of the models considered, this rule for choosing the block length leads to unfeasible results especially for small sample sizes. For instance, even for $n=512$ this rule delivers for Model II block lengths of around $400$. For this reason, we  also consider the moving block bootstrap with nonrandom block length given by $l=n^{1/3}$. This procedure is denoted by \emph{BB}.

 As mentioned in Section \ref{sec.bootval}, a better finite sample performance may be obtained by using bootstrap approximations of studentized statistics. Thus, we consider for the sample mean the statistic $\Xn(2 \pi \fn(0))^{-1/2}$, where $\fn$ is the same SD estimator as the one used for \emph{SDDB}. The sample autocorrelation is studentized as well, where the variance is estimated by Bartlett's formula, \citet[Theorem 7.2.1]{BrockwellDavis1991}, based on the autocorrelation function corresponding to the estimated SD $\fn$. Finally, a standard normal distribution is considered as a further competitor for the studentized statistics and is denoted in the following by \emph{ND}. For non-studentized statistics a normal distribution is used with the variance estimated by using the \emph{SDDB} procedure. Studentization brings clear improvements for all models and all statistics considered. Hence, the focus is on the studentized case and the non-studentized tables can be found in the Supplementary Material.

The coverage probabilities for the studentized sample mean are displayed in Table~\ref{table.mean.ST}. % and for the non-studentized case in Table \ref{table.mean.NST}. Studentization brings clear improvements for all models and all statistics considered. Focusing on the studentized case, 
As it is seen from Table \ref{table.mean.ST}, none of the competitors outperforms the \emph{SDDB} procedure. In fact, in many cases the \emph{SDDB} performs best. Finally, and for Model III it seems that only the \emph{SDDB} procedure gives reasonable estimates. Notice that the SD of Model III is not bounded away from zero, that is, it is not clear whether the \emph{LPB} or the \emph{ARS} are valid in this case. The coverage probabilities for the studentized sample autocorrelation are displayed in Table~\ref{table.rho.ST}. % and in Table \ref{table.rho.NST} for the non-studentized case. Even in the non-studentized case, the coverage probabilities  of \emph{SDDB} are close the nominal levels. Again, more accurate coverage probabilities can be achieved by using studentization.
For this statistic over all, the most accurate coverage probabilities %obtained for the sample autocorrelation case 
are those obtained by using the \emph{ARS} and the \emph{SDDB} procedures.

Notice that block bootstrap methods have their strength in their general applicability, i.e., they are applicable not only to linear processes, like those considered in the simulation study, and to a broad class of statistics. Consequently, it is not surprising that these methods do not perform best for the linear processes considered. %It is therefore, in some sense, an unfair competition between the linear bootstrap approaches and the block-related bootstrap approaches. 
%The fact that block bootstrap methods are outperformed by linear bootstrap methods justifies the latter methods.

Summarizing our numerical findings, it seems that the SDDB performs very good in all model situations and for both statistics considered. In combination with a flexible SD estimator, like for instance the pre-whitening based estimator used in the simulations, the SDDB seems to be a valuable tool for bootstrapping time series. %the first choice to linearly bootstrap the sample mean for linear processes. For the sample autocorrelation, the AR-sieve bootstrap also performs quite well. However, it is not that flexible, that is why we would also recommend the SDDB for this setting.

\begin{table}
%\centering
\caption{Coverage probabilities (in percent) for the mean using the studentized statistic of $\Xn(2 \pi \fn)^{-1/2}$ and for a sample size $n=128$ \label{table.mean.ST}}
\begin{tabular}{|l|ccc|ccc|ccc|}
\hline
           & \multicolumn{3}{c|}{Model I}                                              & \multicolumn{3}{c|}{Model II}                                             & \multicolumn{3}{c|}{Model III}                                            \\
$(1-\alpha)100$ & 80.0       & 90.0      & 95.0      & 80.0       & 90.0       & 95.0      & 80.0       & 90.0       & 95.0       \\
\hline
SDDB       & 78.0                   & 87.1                   & 92.2                    & 78.1                   & 88.7                   & 94.3                    & 80.2                   & 90.0                   & 94.8                    \\
LPB        & 76.0                   & 85.5                   & 90.8                    & 78.2                   & 88.1                   & 92.0                    & 35.1                   & 48.1                   & 61.8                    \\
TBB        & 66.9                   & 77.0                   & 83.0                    & 39.4                   & 46.9                   & 52.2                    & 49.1                   & 56.1                   & 62.7                    \\
ND         & 67.8                   & 78.2                   & 84.6                    & 64.1                   & 76.4                   & 84.2                    & 24.2                   & 32.4                   & 40.0                    \\
ARS        & 76.6                   & 85.9                   & 91.1                    & 74.2                   & 85.5                   & 92.3                    & 39.5                   & 51.8                   & 62.4                    \\
BB         & 28.4                   & 41.0                   & 49.6                    & 30.4                   & 41.2                   & 50.8                    & 34.7                   & 41.6                   & 47.3                    \\ \hline
\end{tabular}

\end{table}

\begin{table}
\caption{Coverage probabilities (in percent) for the lag $2$ autocorrelation using the studentized empirical autocorrelation at lag $2$ and for a sample size $n=128$ \label {table.rho.ST}}
\begin{tabular}{|l|ccc|ccc|ccc|}
\hline
           & \multicolumn{3}{c|}{Model I} & \multicolumn{3}{c|}{Model II} & \multicolumn{3}{c|}{Model III} \\
$(1-\alpha)100$ & 80.0       & 90.0      & 95.0      & 80.0       & 90.0       & 95.0      & 80.0       & 90.0       & 95.0       \\ \hline
SDDB       & 82.5     & 91.5    & 96.0    & 79.4     & 89.3     & 93.5    & 81.0     & 90.3     & 95.3     \\
LPB        & 85.5     & 94.2    & 97.2    & 92.0     & 95.9     & 97.2    & 82.8     & 91.9     & 96.7     \\
TBB        & 76.3     & 83.9    & 87.7    & 20.6     & 25.1     & 27.6    & 65.4     & 72.6     & 77.2     \\
ND         & 75.8     & 87.3    & 92.3    & 73.4     & 84.4     & 90.0    & 79.4     & 88.7     & 94.0     \\
ARS        & 81.8     & 91.3    & 95.9    & 80.8     & 88.9     & 93.2    & 82.0     & 91.4     & 95.7     \\
BB         & 32.6     & 45.0    & 57.0    & 21.4     & 31.2     & 42.8    & 27.9     & 40.0     & 52.2     \\ \hline
\end{tabular}
\end{table}

\subsection{A Real-Life Data Example}
We consider the time series of annual measurements of the water level, in feet, of Lake Huron; cf. Series A in the Appendix of \cite{BrockwellDavis1991} or in  the \emph{R}-package \emph{datasets::LakeHuron}, \cite{R}. Figure \ref{spec.Lake.Huron} shows the results of the following five SD estimators applied to this time series: An AR-pre-whitened, nonparametric estimator of the spectral  density, denoted by \emph{Pre-Whitening}, where the order of the  AR part has been selected  by AIC and the truncation lag  by cross-validation%, \citeasnoun{beltrato1987determining}
; a nonparametric SD estimator using cepstrum thresholding, denoted by \emph{Cepstrum}, see \cite{Stoica2006}; %a nonparametric approach using thresholding the correlogram, denoted by \emph{Thres. Correl.}, see \cite{paparoditis2012nonlinear}; 
a lag window estimator with a  trapezoid kernel and the truncation rule as in \cite{politis2003adaptive}, denoted by \emph{Trapezoid} and an AR parametric approach, where the order of the  AR part has been selected  by AIC. Although, all estimators have a more or less similar overall behavior, they are different with the AR based approaches possessing a stronger peak at frequency zero than the other. %We do not discuss which SD estimator may be best for this setting. Instead, we point out 
We next discuss the impact of these different estimators on the SDDB.  As mentioned in Section \ref{sec.boot}, the SDDB can be either used with the MA or with the AR representation of the process corresponding to the SD estimator applied. Table \ref{table.coef} shows for each estimator the obtained MA-coefficients and AR-coefficients, respectively. As it is seen, depending on the SD estimator used, the MA or the AR representation describes the structure of the process more effectively, i.e., less non-zero coefficients are needed. Clearly, the differences between the SD estimators used manifest themselves in the MA or the AR coefficients obtained. %; it seems that the AR-coefficients of the \emph{Thres. Correl.} SD estimator  are too moderate to describe the peak at frequency zero.
 Notice that the oscillation of the \emph{Trapezoid} SD estimator can be also seen in the behavior of the MA coefficients and that this estimator implicitly fits an MA model to the time series at hand. %It might be argued that the bandwidth for \emph{ Trapezoid} is too small, thus the resulting SD is oscillating too much. This is also displayed in the MA-coefficients, since this estimator fits implicitly an MA model and can be effectively expressed by using the MA coefficients. Nevertheless, we have more then $10$ non-zero coefficients by $n=98$ observations. 

{As this example demonstrates, the broad literature to SD estimation offers a variety of techniques to estimate this function.
%Naturally, there is not the one SD estimator which outperforms all others in any situation. Instead, some situations may require special techniques. That is why such a broad literature to SD estimator exists. 
Therefore, the model used to generate the bootstrap pseudo observations depends on the SD estimator used and which is preferred by the practitioner. The resulting MA or AR representation can then applied to generate pseudo observations in order to bootstrap some statistic of interest.}

\begin{figure}[h]
\includegraphics[width=14cm]{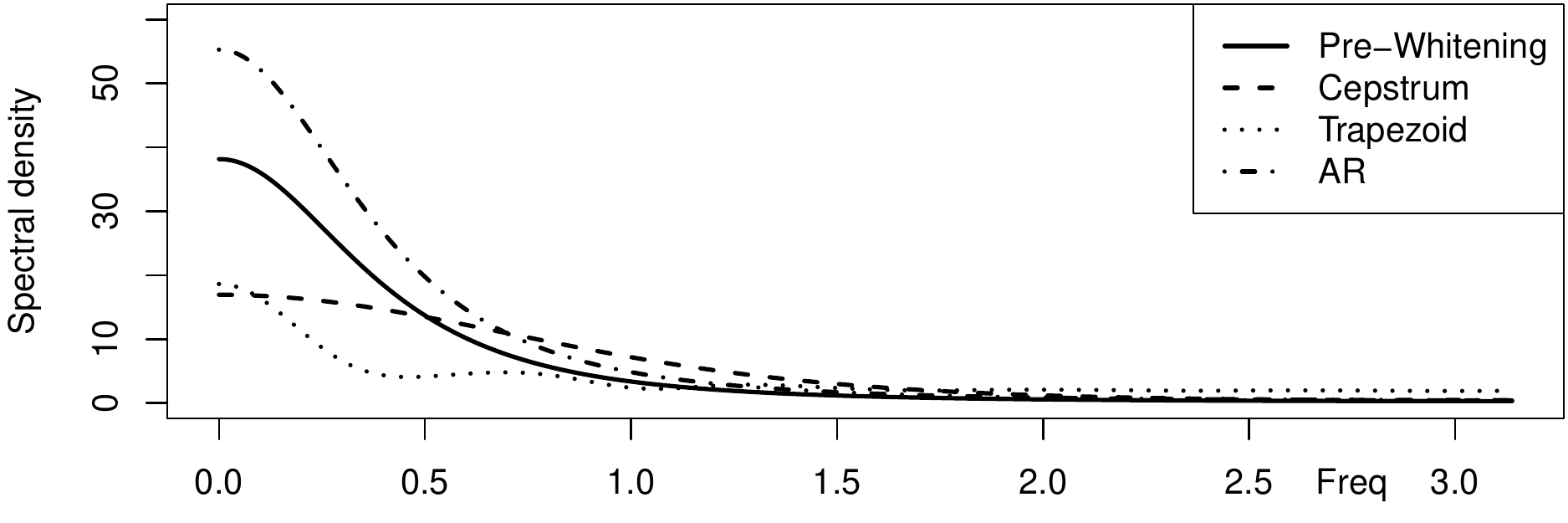}
\caption{Different SD estimates for the Lake Huron data} \label{spec.Lake.Huron}
\end{figure}

\begin{table}
\caption{The MA ($c_{k,n}$) and AR ($b_{k,n}$) coefficients $, k=2,\dots,11$ for the different SD estimates shown in Figure \ref{spec.Lake.Huron} \label{table.coef}}
    \begin{tabular}{l|cccccccccc}
    \hline
    $c_{k,n}$ & 2    & 3     & 4     & 5     & 6    & 7     & 8     & 9    & 10   & 11    \\ \hline
    AR             & 1.05 & 0.84  & 0.61  & 0.42  & 0.28 & 0.18  & 0.12  & 0.07 & 0.05 & 0.03  \\
    Pre-Whitening  & 1.07 & 0.85  & 0.61  & 0.42  & 0.28 & 0.18  & 0.12  & 0.07 & 0.05 & 0.03  \\
    Cepstrum       & 0.93 & 0.43  & 0.14  & 0.03  & 0.01 & 0     & 0     & 0    & 0    & 0     \\
%    Thres. Correl. & 0.39 & 0.40  & 0     & 0     & 0    & 0     & 0     & 0    & 0    & 0     \\
    Trapezoid      & 0.33 & 0.23  & 0.16  & 0.12  & 0.11 & 0.09  & 0.08  & 0.09 & 0.12 & 0.09  \\
\hline\hline
    $b_{k,n}$ & 2    & 3     & 4     & 5     & 6    & 7     & 8     & 9    & 10   & 11    \\ \hline
    AR             & 1.05 & -0.27 & 0     & 0     & 0    & 0     & 0     & 0    & 0    & 0     \\
    Pre-Whitening  & 1.07 & -0.29 & 0.01  & 0     & 0    & 0     & 0     & 0    & 0    & 0     \\
    Cepstrum       & 0.93 & -0.43 & 0.14  & -0.03 & 0.01 & 0     & 0     & 0    & 0    & 0     \\
%    Thres. Correl. & 0.39 & 0.25  & -0.25 & 0     & 0.10 & -0.04 & -0.03 & 0.03 & 0    & -0.01 \\
    Trapezoid      & 0.33 & 0.12  & 0.05  & 0.03  & 0.03 & 0.02  & 0.02  & 0.03 & 0.05 & 0.01  \\
    \end{tabular}
\end{table}
%
%
% \label{table.ma}} 
%\centering
%    \begin{tabular}{l|cccccccccc}
%
%    $c_{k,n}$ & 2    & 3    & 4    & 5    & 6    & 7    & 8    & 9    & 10   & 11   \\
%    \hline
%    AR             & 1.05 & 0.84 & 0.61 & 0.42 & 0.28 & 0.18 & 0.12 & 0.07 & 0.05 & 0.03 \\
%    Pre-Whitening  & 1.07 & 0.85 & 0.61 & 0.42 & 0.28 & 0.18 & 0.12 & 0.07 & 0.05 & 0.03 \\
%    Cepstrum       & 0.93 & 0.43 & 0.14 & 0.03 & 0.01 & 0    & 0    & 0    & 0    & 0    \\
%    Thres. Correl. & 0.39 & 0.40 & 0    & 0    & 0    & 0    & 0    & 0    & 0    & 0    \\
%    Trapezoid      & 0.33 & 0.23 & 0.16 & 0.12 & 0.11 & 0.09 & 0.08 & 0.09 & 0.12 & 0.09 \\
%    \end{tabular}
%
%\end{table}
%
%\begin{table}
%\caption{The AR coefficients $b_{k,n}, k=2,\dots,11$ for the different SD estimates shown in Figure \ref{spec.Lake.Huron} \label{table.ar}} 
%\centering
%    \begin{tabular}{l|cccccccccc}
%
%    $b_{k,n}$ & 2    & 3     & 4     & 5     & 6    & 7     & 8     & 9    & 10   & 11    \\ 
%    \hline
%    AR             & 1.05 & -0.27 & 0     & 0     & 0    & 0     & 0     & 0    & 0    & 0     \\
%    Pre-Whitening  & 1.07 & -0.29 & 0.01  & 0     & 0    & 0     & 0     & 0    & 0    & 0     \\
%    Cepstrum       & 0.93 & -0.43 & 0.14  & -0.03 & 0.01 & 0     & 0     & 0    & 0    & 0     \\
%    Thres. Correl. & 0.39 & 0.25  & -0.25 & 0     & 0.10 & -0.04 & -0.03 & 0.03 & 0    & -0.01 \\
%    Trapezoid      & 0.33 & 0.12  & 0.05  & 0.03  & 0.03 & 0.02  & 0.02  & 0.03 & 0.05 & 0.01  \\
%    \end{tabular}
%\end{table}

\section{Conclusions}
\label{conclusions}
In this paper a spectral  density factorization has been used to obtain consistent  estimates of the entire sequence of moving average coefficients  of the Wold representation of a stationary nondeterministic  process.  A bootstrap procedure then has been proposed which 
uses the estimated sequence of moving average coefficients together with  a sequence of pseudo innovations to generate new pseudo time series. Apart for the choice of the pseudo innovations, this bootstrap procedure is completely driven and controlled by the spectral density estimator used. For i.i.d. pseudo innovations the new bootstrap method generalizes existing linear bootstrap methods, like for instance, the AR-sieve bootstrap. The latter is a special case of the spectral density-driven bootstrap, which is obtained if an autoregressive spectral density estimator is used. We established asymptotic validity of the proposed bootstrap method driven by i.i.d. pseudo innovations for  linear processes and for interesting classes of statistics. The good finite sample behavior of the new bootstrap method has been demonstrated by means of simulations.

\section{Proofs}
\label{sec:proofs}

\begin{proof}[Proof of Theorem \ref{consistence}]
a) Since $\fn$ is a uniformly consistent estimator, it follows that some function $g$, with $g(n) \to \infty$ as $n \to \infty$, exists such that
$
\sup_{\lambda \in [-\pi,\pi]} | \fn (\lambda) - f (\lambda)| = o_P(1)= O_P\left({g(n)}^{-1}\right).
$
Consequently, we have that for all $\eps>0$, there exists a $\Omega_0 \in \mathcal{A}$ with $P(\Omega_0)\geq 1-\eps$ and a $n_0 \in \N$, such that for all $\omega \in \Omega_0$ a constant $C>1$  exits, such that for all $n\geq n_0$ it holds true that  %$ \forall \eps > 0 \exists \Omega_0 \in \mathcal{A}$ with $P (\Omega_0) \geq 1-\eps$ and $n_0 \in \N \forall  \omega \in \Omega_0 \exists C \in \R \forall n  \geq n_0 : $ 
$\sup_{\lambda \in [-\pi,\pi]} | \fn (\lambda) - f (\lambda)|  \leq {C}{g(n)}^{-1}.$
Since $\log f$ and $\log \fn$ are integrable, and the set $\{\lambda \in [-\pi,\pi) : f(\lambda)=0 \text{ or } \fn(\lambda)=0 \}=:B_0^{\mathrm C}$ is a null set, we have
%\begin{align*}
$\hat a_{k,n}  =%&\frac{1}{2 \pi}\int_{-\pi}^\pi \log ( \fn (\lambda)) \exp(-ik\lambda) d \lambda=\frac{1}{2 \pi} \int_{-\pi}^\pi \left(\log ( \fn (\lambda)) + \log f (\lambda) - \log f (\lambda)\right) \exp(-ik\lambda) d \lambda \\
 %=& a_k +\!\!\int_{-\pi}^\pi \left[\log ( \fn\lambda) - \log f (\lambda)\right] \exp(-ik\lambda)   \frac{\dlambdaB}{2\pi}=
  a_k +\!\!\int_{B_0}\left[\log ( \fn\lambda) - \log f (\lambda)\right] \exp(-ik\lambda)   \frac{\dlambdaB}{2\pi}.
 $
%\end{align*}
We get further \\
%\begin{align*}
$
 2\pi\sup_{k \in \N} |\hat a_{k,n}-a_k| = \sup_{k \in \N} \left|  \int_{B_0}  \left(\log ( \fn (\lambda)) - \log f (\lambda)\right) \exp(-ik\lambda)  \dlambdaB\right| \\
 \leq %\int_{B_0} | \log ( \fn (\lambda)) - \log( f (\lambda)) | \dlambdaB \\
  %&=
   \int_{B_0}\!\!\! \ind_{\{\fn(\lambda) > f (\lambda)\}} (\lambda) \log \left({\fn (\lambda)}/{f (\lambda) }\right)d \lambda 
  +\int_{B_0} \!\!\! \ind_{\{f(\lambda) > \fn (\lambda)\}} (\lambda)  \log \left({f (\lambda)}/{ \fn (\lambda) } \right)  d \lambda.
$
   %\end{align*}
Let $\omega \in \Omega_0$ and consider the case $\fn(\lambda)> f (\lambda)>0$. Then we have $\fn(\lambda) \leq {C}{g(n)}^{-1} + f(\lambda)$. Consequently, 
% \begin{align*}
$
 \int_{B_0} \ind_{\{\fn(\lambda) > f (\lambda)>0\}} (\lambda)  \log \left(\!{\fn (\lambda)}/{f (\lambda) } \!\right)\! d \lambda \leq  %\int_{B_0 \cup \{\fn > f\}} \log \left(\frac{C / g(n) + f(\lambda)}{f (\lambda) }\right) d \lambda \\
% =&
%\int_{B_0 \cup \{\fn > f\}}  \log \left(\!{C}/({g(n) f (\lambda)}\!) +1 \right) d \lambda.=
\int_{B_0 \cup \{\fn > f\}}  \log(f(\lambda)+C/g(n))-\log(f\lambda) d \lambda.
$
% \end{align*}
 Assume $g(n)>1$. Then it holds true for all $n \in \N$ and all $\lambda \in [-\pi,\pi]$ that 
$
|\log \left(C(g(n) f (\lambda))^{-1} +1 \right)| \leq %|\log\left( {C}/{f(\lambda)} +1 \right)| %= \log\left( \frac{C + f (\lambda)}{f (\lambda)}\right) \\
%= |\log ( C + f (\lambda))- \log ( f (\lambda)) |\leq 
|\log ( C + f (\lambda)) | +   |\log ( f (\lambda)) |.
$
Since $\log f$ is integrable, $\log (C+ f (\cdot))$ is integrable as well, and consequently, the dominated convergence theorem can be applied. We get $\lim_{ n \to \infty}\int_{B_0 \cup \{\fn > f\}}  \log(f(\lambda)+C/g(n))-\log f(\lambda) d \lambda=0$. %\int_{B_0 \cup \{\fn > f\}}\lim_{ n \to \infty}  \log(f(\lambda)+C/g(n))-\log f(\lambda) d \lambda=0.$$ 
%\\ = &\int\limits_{\{\lambda \in [-\pi,\pi] : f(\lambda) >0, \fn(\lambda) > f (\lambda)\}} \lim_{ n \to \infty} \log \left(\frac{C}{g(n) f (\lambda)} +1 \right) d \lambda \\
%=  &\int\limits_{\{\lambda \in [-\pi,\pi] : f(\lambda) >0, \fn(\lambda) > f (\lambda)\}} \lim_{ n \to \infty}  \log 1 d \lambda = 0
%\end{align*}
Analogously, it can be shown that $\lim_{n \to \infty} \int_{B_0 \cup \{\fn < f\}} \log \left({f (\lambda)}/{ \fn (\lambda) } \right)d \lambda=0$. % \leq \lim_{n \to \infty} \int_{B_0 \cup \{\fn < f\}} \log(1- C/(g(n) f(\lambda))) \dlambda=0.$$
Thus, we have for all $\omega \in \Omega_0$ that $\sup_{k \in \N} |{1}/({2 \pi})  \int_{-\pi}^\pi \log ( \fn (\lambda)) - \log f (\lambda) \exp(-ik\lambda) d \lambda|\to 0$ as $n \to \infty$. This proves assertion a). 
%$\sup_{k \in \N} |\frac{1}{2 \pi}  \int_{-\pi}^\pi \log ( \fn (\lambda)) - \log f (\lambda) \exp(-ik\lambda) d \lambda|\to 0 , n \to \infty=o_P(1)$ and the assertion a) follows. \newline \\
For b) and  c) fix a $k \in \N$ and observe that $\hat c_{k,n}$ is a continuous transformation of a finite number of $\hat a_{k,n}$'s. Thus, $\sup_{k \in \N} |\hat a_{k,n}-a_k | = o_P(1)$ ensures, that $\hat c_{k,n} \overset{P}{\to} c_k$ as $n \to \infty$. The same arguments apply to $\hat b_{k,n}$.
\end{proof}

Before proving Theorem \ref{maseries}, we notice the following useful lemma with proof in the Supplementary Material.
\begin{lemma}\label{assumptions}
Let the condition of part b) of Theorem \ref{maseries} be satisfied. % SD $f$ and its estimator $\fn$ fulfill
%\begin{align}
%\forall n \in \N \, \forall \lambda \in [-\pi,\pi] : f(\lambda), \fn(\lambda) \in [C_1,C_2] , \text{ where } C_1>0, C_2 < \infty \label{bounded},\\
%\sup_{\lambda \in [-\pi,\pi]} | f(\lambda)-\fn(\lambda)| = o_P(1) \label{f_const},\\
%\sup_{\lambda \in [-\pi,\pi]} \left|\frac{\dx}{\dx \lambda} f(\lambda)-\frac{\dx}{\dx \lambda}\fn(\lambda)\right| = o_P(1), \label{df_const}.
%\end{align}
%Furthermore, the first derivates of $f$ and $\fn$ are continuous and integrable. 
Then it holds true that
a) $
\sup_{\lambda \in [-\pi,\pi]} | {f}^{1/2}(\lambda) - {\fn}^{1/2}(\lambda)| = o_P(1)$, b) $
\sup_{\lambda \in [-\pi,\pi]} | \log \left(f(\lambda)\right)-\log(\fn(\lambda))| = o_P(1)$, and c) $
\sup_{\lambda \in [-\pi,\pi]}  \left|\frac{\dx}{\dx \lambda} \log (f(\lambda))-\frac{\dx}{\dx \lambda} \log\left(\fn(\lambda)\right)\right| = o_P(1) $
%\end{align}
\end{lemma}
%\begin{proof}[Proof of Lemma 6.1]%\ref{assumptions}]
%Since $(\ref{bounded})$ ensures that $f$ and $\fn$ are bounded away from zero, $(\ref{lem.6.1.1})$ and $(\ref{lem.6.1.2})$ follow immediately from the mean value theorem. To see why the third assertion is true, one considers the following bound
%\begin{align*}
%&\sup_{\lambda \in [-\pi,\pi]} \left|\frac{\dx}{\dx \lambda} \log \left(f(\lambda)\right)-\frac{\dx}{\dx \lambda} \log\left(\fn(\lambda)\right)\right|  \\
%%& \leq \sup_{\lambda \in [-\pi,\pi]} \left|\frac{\frac{\dx}{\dx \lambda} f (\lambda)}{f( \lambda)} - \frac{\frac{\dx}{\dx \lambda} \fn (\lambda)}{f( \lambda)}\right| +\sup_{\lambda \in [-\pi,\pi]} \left|\frac{\frac{\dx}{\dx \lambda} \fn (\lambda)}{f( \lambda)} -\frac{\frac{\dx}{\dx \lambda} \fn (\lambda)}{\fn( \lambda)}\right| \\
%& \leq \sup_{\lambda \in [-\pi,\pi]}  \left| \frac{\dx}{\dx \lambda} f(\lambda)- \frac{\dx}{\dx \lambda} \fn(\lambda)\right|/{f(\lambda)} + \sup_{\lambda \in [-\pi,\pi]} {\frac{\dx}{\dx \lambda} f (\lambda)}  \frac{ | \fn(\lambda) - f (\lambda)|}{f(\lambda) \fn(\lambda)}.
%\end{align*}
%Since $f$ and $\fn$ are bounded away from zero and $\frac{\dx}{\dx \lambda} f$ is finite, $(\ref{lem.6.1.3})$ follows from $(\ref{f_const})$ and $(\ref{df_const})$.
%\end{proof}
\begin{proof}[Proof of Theorem \ref{maseries}]
Since $(a_k)$ and $(\hat a_{k,n})$ are the Fourier coefficients of $\log f$ and $\log \fn$ respectively, we have 
$
\intpi | \log f(\lambda) - \log \fn (\lambda) |^2 \dlambdaB = \intpi | \sum_{k=-\infty}^\infty (a_k -\hat a_{k,n}) \exp( i k \lambda) |^2 \dlambdaB,
$
and consequently (\ref{L2ak}) follows from Parseval's identity. (\ref{L2akk}) follows by the same argument. Using Jensen's inequality and since $\sum_{k=1}^\infty k^{-2}=\pi^2/6$, we get %the following holds true
%\begin{align*}
$
\sum_{k=1}^\infty | a_k-\hat a_{k,n}|  %&= \frac{\pi^2}{6} \left(\frac{6}{\pi^2} \sum_{k=1}^\infty \frac{| a_k-\hat a_{k,n} | k^2}{k^2} \right)^{2/2} 
%\leq \frac{\pi^2}{6} \left(\frac{6}{\pi^2} \sum_{k=1}^\infty \frac{| a_k-\hat a_{k,n} |^2 k^4}{k^2} \right)^{1/2} \\
%&\leq \frac{\pi}{\sqrt{12}} \left(\sum_{k=-\infty }^\infty | a_k-\hat a_{k,n} |^2 k^2 \right)^{1/2} 
\leq \sqrt{\pi/24}\left(\int_0^{2 \pi} |\frac{\dx}{\dx \lambda}  \log f (\lambda) -\frac{\dx}{\dx \lambda}  \log \fn(\lambda)|^2  \dlambdaB \right)^{1/2}\!\!\!.$ %\end{align*}
Lemma \ref{assumptions} implies then, that the above bound converges to zero in probability as $n\to \infty$. 
Let $\sgn(k)=\ind_{\{k>0\}}-\ind_{\{k<0\}}$. Since Lemma A.2 %\ref{A-lemphase}
 (or see \citet{pourahamdi1984}) ensures that for all $\lambda \in [0,2\pi]$ it holds true that ${\sigma}/{\sqrt{2\pi}}\sum_{k=0}^\infty c_k \exp( i k \lambda) = \exp( a_0/2 + \sum_{k=1}^\infty a_k \exp(i k \lambda))$ and similarly for $\{\hat c_{k,n}\}$ with $a_k$ replaced by $\hat a_{k,n}$, we get by Parseval's identity, the fact that $\cos(x) \geq 1-0.5 x^2$ for all $x\in \R$ and $\intpi \sin(k \lambda) \sin(l \lambda) \dlambdaB=\pi \ind_{\{k=l\}}$ for all $k,l \in \N$, that
\begin{align*}
&\sum_{k=0}^\infty |\sigma c_k-\sigmah\hat c_{k,n}|^2 = \intpi \left| \sum_{k=0}^\infty (\sigma c_k-\sigmah \hat c_{k,n}) \exp(i k \lambda)\right|^2 \dlambda \\
%&=\intpi \left| \exp[a_0/2 + \sum_{k=1}^\infty a_k \exp(ik\lambda)] - \exp[a_{0,n}/2 + \sum_{k=1}^\infty\hat a_{k,n} \exp(ik\lambda)]\right|^2 \dlambdaB \\
&=\intpi \left| \exp\left[1/2 \sum_{k=-\infty}^\infty a_k \exp(ik\lambda)+ 1/2 \sum_{k=-\infty}^\infty \sgn(k) a_k \exp(ik\lambda)\right] \right.  \\ 
& \left. \qquad \, \quad -\exp\left[1/2\sum_{k=-\infty}^\infty\hat a_{k,n} \exp(ik\lambda)+1/2 \sum_{k=-\infty}^\infty \sgn(k)\hat a_{k,n} \exp(ik\lambda)\right]\right|^2 \dlambdaB \\ 
%&=\!\intpi \! \left|{ f^{1/2}(\lambda)}\exp \!\left[i\sum_{k=1}^\infty a_k \sin(k\lambda)\right]\! -\!  {\fn^{1/2}\! (\lambda)}\exp\!\left[i\sum_{k=1}^\infty\hat a_{k,n} \sin(k\lambda)\right]\right|^2 \!\!\dlambdaB \\
%&=\intpi \left|\exp\left[i\sum_{k=1}^\infty\hat a_{k,n} \sin(k\lambda)\right] \right|^2\left|{ f^{1/2}(\lambda)} \exp\left[i\sum_{k=1}^\infty (a_k-\hat a_{k,n} ) \sin(k\lambda)\right] -  {\fn}^{1/2} (\lambda)\right|^2 \dlambdaB \\
&=\intpi \left(f(\lambda) +\fn(\lambda) -2 f^{1/2}(\lambda) \fn^{1/2}(\lambda) \cos\left[\sum_{k=1}^\infty (a_k-\hat a_{k,n}) \sin(k \lambda)\right]\right) \dlambdaB \\
%&\leq \intpi \!\!\! \left(f^{1/2}(\lambda)-\fn^{1/2}(\lambda)\right)^2 \!\!\!\dlambdaB +   \intpi \!\! (f(\lambda) \fn(\lambda))^{1/2}\left|\sum_{k=1}^\infty (a_k -\hat a_{k,n}) \sin(k\lambda)\right|^2 \dlambdaB \\
&\leq \intpi \!\!\! \left(f^{1/2}(\lambda)-\fn^{1/2}(\lambda)\right)^2 \!\!\!\dlambdaB +  \!  \! \! \sup_{\lambda \in [0,2\pi]}  \!\! \! (f(\lambda) \fn(\lambda))^{1/2} \intpi \!\! (\log f (\lambda) - \log \fn (\lambda))^2 \dlambdaB= \text{o}_P(1)
%\\&{\color{red}= \intpi \left(f^{1/2}(\lambda)-\fn^{1/2}(\lambda)\right)^2 \dlambdaB  + \mathcal{O}_P(1) \intpi (\log f (\lambda) - \log \fn (\lambda))^2 \dlambdaB } 
%\\&=\intpi f(\lambda)\left| \exp[i\sum_{k=1}^\infty (a_k -\hat a_{k,n} ) \sin(k\lambda)] -1 +1 - \left(\frac{\fn (\lambda)}{f(\lambda)}\right)^{1/2}\right|^2 \dlambda \\
%&\leq \intpi 2f(\lambda)\left| \exp[i\sum_{k=1}^\infty (a_k -\hat a_{k,n} ) \sin(k\lambda)] -1\right|^2 \dlambda + 2\intpi\left| {f ^{1/2}(\lambda)} - {\fn}^{1/2}(\lambda)\right|^2 \dlambda.
\end{align*}

where the last equation follows by $(\ref{L2ak})$ and Assumption 2. %, $(\ref{bounded})$ and $(\ref{f_const})$.
Assertion  (\ref{L2ck}) follows since $\sigmah  \stackrel{P}{\rightarrow} \sigma$ and
$
\sum_{k=0}^\infty |c_k-\ckn|^2 \leq {2}/{\sigma} \sum_{k=0}^\infty | \sigma c_k- \sigmah \ckn|^2+ {2}/{\sigma} \sum_{k=0}^\infty |\ckn|^2 |\sigmah-\sigma|^2=\text{o}_P(1).%+\text{o}_P(1)
$ By Jensen's inequality, we have %the $L_1$-norm of the series $(\sigma c_k-\sigmah \hat c_{k,n})_{k\in \N}$ can be bounded by the $L_2$-norm of the series $(\sigma c_k- \sigmah \hat c_{k,n})_{k\in \N}$. Consequently it holds
\begin{align*}
&\sum_{k=1}^\infty |\sigma c_k-\sigmah \hat c_{k,n}|  \leq %\frac{\pi^2}{6} \left( \frac{6}{\pi^2} \sum_{k=1}^\infty |\sigma c_k- \sigmah \hat c_{k,n}|^2 k^2 \right)^{1/2} \\
%&=\frac{\pi^2}{6} \left( \frac{6}{\pi^2} \intpi \left| \sum_{k=0}^\infty (\sigma c_k-\sigmah \hat c_{k,n}) k \exp(i k \lambda) \right|^2 \dlambda \right)^{1/2} \\&=
%\frac{\pi}{\sqrt{6}} \left(  \intpi \left| \frac{\dx}{\dx \lambda} \sum_{k=0}^\infty (\sigma c_k-\sigmah \hat c_{k,n}) \exp(i k \lambda) \right|^2 \dlambdaB \right)^{1/2} \\
%&=\frac{\pi}{\sqrt{6}}\! \left(\intpi \left| \frac{\dx}{\dx \lambda} \left(\! \exp\!\left[\frac{a_0}{2}+ \sum_{k=1}^\infty a_k \exp(i k \lambda)\right]-\exp\!\left[\frac{\hat a_{0,n}}{2}+\sum_{k=1}^\infty\hat a_{k,n} \exp(ik\lambda)\right]\right)\right|^2 \!\! \dlambdaB\!\right)^{1/2}  \\
%&=\frac{\pi^2}{6} \left( \frac{6}{\pi^2} \intpi \left| \frac{\dx}{\dx \lambda} \left( \exp[a_0/2+ \sum_{k=1}^\infty a_k \exp(i k \lambda)] \right.\right.\right. \times \\
%& \qquad\left.\left.\left.\left[1-\exp[a_{0,n}/2-a_k/2+\sum_{k=1}^\infty (\hat a_{k,n}-a_k) \exp(ik\lambda)]\right]\right)\right|^2 \dlambdaB \right)^{1/2} \\
\frac{\pi}{\sqrt{6}} \left( \intpi \left| \exp\left[ \frac{\hat a_{0,n}}{2} + \sum_{k=1}^\infty (\hat a_{k,n}) \exp(ik\lambda)\right] \sum_{k=1}^\infty (\hat a_{k,n}-a_k) (k i) \exp(ik \lambda) -\right.  \right. \\
&\left. \left.  \sum_{k=1}^\infty (k i) a_k \exp( i k \lambda)  \left(\!\! \exp\!\!\left[\frac{a_0}{2}+\sum_{k=1}^\infty a_k \exp(ik \lambda) \right] - \exp\left[\frac{\hat a_{0,n}}{2}+ \sum_{k=1}^\infty (\hat a_{k,n}) \exp(ik \lambda)\right]\right)\right|^2 \dlambdaB \right)^{1/2}\!\!\!\!.
%\end{align*}
%Now, by the triangle inequality and since $(\hat a_{k,n})$ are Fourier coefficients of $\fn$, it holds:
%\begin{align*}
%\sum_{k=1}^\infty& |\sigma c_k-\sigmah \hat c_{k,n}|  
\end{align*}
The term on the right hand side of the last equation can be bounded by \\
%\begin{align*}
$
\frac{\pi}{\sqrt{3}} \left(\intpi \left[ \fn(\lambda) \left| \sum_{k=1}^\infty (\hat a_{k,n}-a_k) (k i) \exp(ik \lambda) \right|^2 + \left| \sum_{k=1}^\infty (k i) a_k \exp( i k \lambda) \right|^2 \right. \right. \\
\left.\left. \times   \left|  \exp\left[{a_0}/{2}+\sum_{k=1}^\infty a_k \exp(ik \lambda) \right] - \exp\left[{a_{0,n}}/{2}+ \sum_{k=1}^\infty (\hat a_{k,n}) \exp(ik \lambda)\right]\right|^2  \right]\dlambdaB   %\right. 
%\\& \qquad\left.+ 2 \intpi \fn(\lambda) \left| \sum_{k=1}^\infty (\hat a_{k,n}-a_k) (k i) \exp(ik \lambda) \right|^2 \dlambdaB 
\right)^{1/2}\!\!\! .$
%\end{align*}

Furthermore, the second part of the last term can be bounded analogously as in the case for $(\ref{L2ck})$  and because of Assumption 2, we get %  
\begin{align*}
&\sum_{k=1}^\infty |\sigma c_k-\sigmah \hat c_{k,n}|  \leq 
 \frac{C \pi}{\sqrt{3}} \sup_{\lambda \in [0,2\pi]} \left| \sum_{k=1}^\infty a_k k \exp(ik\lambda) \right| \left( \frac{1}{4} \intpi  \left| \frac{\dx}{\dx \lambda}{\log f (\lambda)} - \frac{\dx}{\dx \lambda} \log {\fn}(\lambda) \right|^2 \dlambdaB \right.\\
& \qquad\left. +  \sum_{k=1}^\infty (\hat a_{k,n} -a_k) ^2k^2  \intpi \left|{f ^{1/2}(\lambda)} - {\fn}^{1/2}(\lambda) \right|^2 \dlambdaB \right)^{1/2} \\
%&= \left(\mathcal{O}_P(1) \intpi \left(\frac{\dx}{\dx \lambda} \log f(\lambda)-\frac{\dx}{\dx \lambda} \log \fn(\lambda) \right)^2 \dlambdaB + \mathcal{O}_P(1) \intpi \left|{f ^{1/2}(\lambda)} - {\fn}^{1/2}(\lambda) \right|^2 \dlambdaB\right)^{1/2}.
&= \left(\mathcal{O}_P(1) \intpi \left| \frac{\dx}{\dx \lambda} \log f(\lambda)-  \frac{\dx}{\dx \lambda}\log \fn(\lambda) \right|^2 \dlambdaB + \mathcal{O}_P(1) \intpi \left|{f ^{1/2}(\lambda)} - {\fn}^{1/2}(\lambda) \right|^2 \dlambdaB\right)^{1/2}\!\!\!\!.
\end{align*}
The convergence to zero in probability, as $n \to \infty$, of the term on the right hand side of the last equality follows from $(\ref{L2akk})$  and the fact that $\sigmah  \stackrel{P}{\rightarrow} \sigma$. %Assertion $(\ref{L2ckk})$ %is a byproduct of these calculations which follows then immediately. 
\end{proof}

%\begin{lemma}\label{summable}
%If the Fourier coefficients of $\log f$ fulfill \linebreak$a_k=\int_{-\pi}^\pi \log ( f (\lambda)) \exp(-ik\lambda)  \dlambda =\mathcal{O}(k^{-3})$, then the coefficients $(c_j)_{j \geq 0}$ satisfy $c_j=\mathcal{O}(j^{-3})$ and therefore $\sum_{j=0}^\infty  |j c_j | < \infty$. \\If $\log f$ is twice continuous differentiable and the second derivate is of bounded variation, the condition $a_k=\mathcal{O}(k^{-3})$ is fulfilled, see \cite[Ch. 2, Theorem 4.12]{zygmund2002}.
%\end{lemma}
%\begin{proof}
%Follows by equation (\ref{eq.ck}).
%\end{proof}
%Summability properties of the Wold's coefficients can also be directly derived from summability properties of the autocovariance function. %\citeasnoun[Th. 2.13 and Th. 2.14]{jpksp} give the following result. 
%Let $\{X_t : t \in \Z\}$ be a purely nondeterministic stationary time series with autocovariance function $\gamma$ fulfilling $\sum_{k \in \Z} (1+|k|^r)| \gamma(k)|<\infty$ and furthermore the SD is bounded away from zero. Then the coefficients of the Wold representation $\{c_k : k \in \N\}$ fulfill $\sum_{k=1}^\infty (1+|k|^r) |c_k| < \infty$. 
\begin{proof}[Proof of Theorem \ref{validity1}]
For simplicity, the proof is stated for centered $X_t^*$'s and a single $\hat T_{n,p}$ and in order to simplify notation the subscript $p$ is in the following omitted. Since $g$ is a smooth function, the delta-method can be applied to get the asymptotic normality of $\hat T_n^*$. Consider the statistic 
$
T_{n}^*%={1}/{n} \sumtn \sum_{s=1}^n b(t-s) X_t^* X_s^* 
= {1}/{n} \sumtn \sum_{h=1-t}^{n-t} d(h) X_t^* X_{t+h}^*.
$
We firstly show that, as  $n \to \infty$,
$
\sqrt{n} ( T_{n}^{*M} - E^* T_{n}^{*M}) \to \mathcal{N} (0,\sigma_M^2)$, in probability,  where $T_{n}^{*M}= \frac{1}{n} \sumtn \sum_{h=1-t}^{M \land (n-t)} d(h) X_t^* X_{t+h}^*.$ For this, we use a central limit theorem for triangular arrays of weakly dependent random variables established by \cite{CLT1}. Let
%\begin{align*}
$
\sqrt{n} ( T_{n}^{*M} - E^* T_{n}^{*M}) = \sumtn \frac{1}{\sqrt n }\sum_{h=1-t}^{M \land (n-t)} d(h) (\Xs{t} \Xs{t+h} - \gammah(h)) = \sumtn Z_{t,n}^*,
$
%\end{align*}
with an obvious notation for $Z_{t,n}^*$ and $\gammah(h)=1/(2\pi)\int_{-\pi}^\pi \fn(\lambda) \exp(-ih\lambda) \dlambdaB$. For the mean of $Z_{t,n}^*$, we have
$
E^* Z_{t,n}^* %=  \frac{1}{\sqrt n }\sum_{h=1-t}^{M \land (n-t)} d(h) ( E ^*\Xs{t} \Xs{t+h} - \gammah(h))
 = 0 \, \text{ for all } t \in \Z.
$
Furthermore, since $\{\Xs t, t \in \Z\}$ is a linear process, its strictly stationarity can be used to show that 
%\begin{align*}
$
\sumtn \Es (Z_{t,n}^*)^2%& = \sumtn \frac{1}{n}  \sum_{h_1=1-t}^{M \land (n-t)} d(h_1) \sum_{h_2=1-t}^{M \land (n-t)} d(h_2) \Es(\left( (\Xs{t})^2 \Xs{t+h_1} \Xs{t+h_2} - \gammah(h_1) \gammah(h_2)\right) \\
\leq % \sumtn \frac{1}{n}  \sum_{h_1=1-t}^{M \land (n-t)} | d(h_1)|  \sum_{h_2=1-t}^{M \land (n-t)} |d(h_2)| \Es\left( 2(\Xs{t})^4 +4(\Xs{t+h_1})^4 +4 (\Xs{t+h_2})^4 + \gammah(h_1) \gammah(h_2)\right) \\&=
%\sumtn \frac{1}{n}  \sum_{h_1=1-t}^{M \land (n-t)} | d(h_1)|  \sum_{h_2=1-t}^{M \land (n-t)} |d(h_2)| \left( 10 \Es(\Xs{1})^4  + \gammah(h_1) \gammah(h_2)\right) \\
%\leq 
(\Es (\Xs{1})^4 + \gammah(0)^2) (\sum_{h \in Z} | d(h)|)^2 \leq C< \infty,$
%\end{align*}
where $C$ is independent of $n$, since $\hat \kappa_{4,n}$ and $\gammah$ are uniformly bounded and so is $\Es (\Xs{1})^4$. \\
Consider the weak dependence structure of $Z_{t,n}^*$. For this,  let $u \in \N$ and consider time points $1 < s_1< \dots< s_u<s_u+r=t_1\leq t_2 \leq n$, a square integrable and measurable function $f : \R^u \to \R$ and a bounded and measurable function $\tilde f : \R^u \to \R$. Without loss of generality, we assume, that $r > M$. Then, we have
%\begin{align*}
%&
$
|\sqrt n \cov(f(Z_{s_1,n}^*  ,\dots,Z_{s_u,n}^*),Z^*_{t_1,n})| = \left|\cov\!\left(f(Z_{s_1,n}^*,\dots,Z_{s_u,n}^*), \sum_{h=1- s_u-r}^{M \land ( n - (s_u+r))} d(h) (\Xs{s_u+r} \Xs{s_u+r+h})\right)\right| \\
%&
=\left|\cov\left(f(Z_{s_1,n}^*,\dots,Z_{s_u,n}^*), \sum_{h=1- s_u-r}^{M \land ( n - (s_u+r))} d(h) \sum_{j=0}^\infty \hat c_{j,n} \sum_{l=0}^\infty \hat c_{l,n} \eps^*_{s_u+r-j} \eps_{s_u+r+h-l}^*\right)\right|.
$
%\end{align*}
Since $\{\Xs t\}$ is an one-sided linear process, we have $f(Z_{s_1,n}^*,\dots,Z_{s_u,n}^*)=$\linebreak$f(g(\eps_{s_u+M}^*,\eps_{s_u+M-1}^*,\dots))$ for some measurable function $g$. Consequently, by the independence of the $\eps^*_t$'s and applying Cauchy-Schwarz's inequality, it follows for the last expression above that it can be bounded by%as follows
\begin{align*}
%&\left| n^{-1/2}\!\!\!\!\!\!\!\!\!\!\!  \sum_{h=1- s_u-r}^{M \land ( n - (s_u+r))}\!\!\!\!\!\!\!\! d(h) \!\!\!\!\sum_{j=r-M}^\infty \!\!\!\!\hat c_{j,n}\!\!\!\! \sum_{l=r+h-M}^\infty\!\! \!\! \hat c_{l,n} \cov\left(f(g(\eps_{s_u+M}^*,\eps_{s_u+M-1}^*,\dots)), \eps^*_{s_u+r-j} \eps_{s_u+r+h-l}^*\right)\right| \\
  &\ n^{-1/2}\!\!\!\!\!\!\!\!\!\!\! \sum_{h=1- s_u-r}^{M \land ( n - (s_u+r))}\! \!\!\! \!\!\! \!\!|d(h)|\! \!\! \!\! \sum_{j=r-M}^\infty \!\!|\hat c_{j,n}|\!\! \!\!\!\!\!\sum_{l=r+h-M}^\infty \!\!\!\!\!\!  \! \!|\hat c_{l,n}| (E^* f^2(Z_{s_1,n}^* ,\dots,Z_{s_u,n}^*))^{1/2} \Big[E^* ( \eps^*_{s_u+r-j} \eps^*_{s_u+r+h-l})^2\Big]^{1/2} \\
\leq& (E^* f^2(Z_{s_1,n}^* ,\dots,Z_{s_u,n}^*))^{1/2}  \frac{1}{\sqrt n } \sum_{h \in \Z} |d(h)|  \sum_{l=0}^\infty |\hat c_{l,n}| \max(\hat \kappa_{4,n},1) \sum_{j=r-M}^\infty |\hat c_{j,n}| \\
  \leq& (E^* f^2(Z_{s_1,n}^* ,\dots,Z_{s_u,n}^*))^{1/2}  \frac{1}{\sqrt n } C  \sum_{j=r-M}^\infty |\hat c_{j,n}| ,%
\end{align*}
where $C< \infty$ is independent of $n$ since $\hat \kappa_{4,n}$ and $ \sum_{j=0}^\infty  \hat c_{j,n}$ are uniformly bounded. We have
\begin{align*}
&|\cov( \tilde f(Z_{s_1,n}^*  ,\dots,Z_{s_u,n}^*),Z^*_{t_1,n} Z^*_{t_2,n})| =\\ 
&= \left|  \sum_{h=1- s_u-r}^{M \land ( n - (s_u+r))} \!\!\!\!\!\!d(h) \!\!\!\sum_{j=r-M}^\infty\!\!\! \hat c_{j,n}\!\!\! \sum_{l=r+h-M}^\infty \!\!\!\hat c_{l,n} \cov(\tilde f(Z_{s_1,n}^*,\dots,Z_{s_u,n}^*),\frac{1}{\sqrt n} \eps^*_{s_u+r-j} \eps^*_{s_u+r+h-l} Z^*_{t_2,n}) \right| \\
%& \leq \sum_{h=1- s_u-r}^{M \land ( n - (s_u+r))}\!\!\!\!\!\! |d(h)|\!\!\! \sum_{j=r-M}^\infty\!\!\! |\hat c_{j,n}|\!\!\! \sum_{l=r+h-M}^\infty |\hat c_{l,n} | 2 \| \tilde f \|_\infty  E^* (\frac{1}{\sqrt n} |\eps^*_{s_u+r-j} \eps^*_{s_u+r+h-l} Z^*_{t_2,n}|)  \\
& \leq \sum_{h=1- s_u-r}^{M \land ( n - (s_u+r))}\!\!\!\!\!\! |d(h)|\!\!\! \sum_{j=r-M}^\infty\!\!\! |\hat c_{j,n}| \!\!\!\sum_{l=r+h-M}^\infty |\hat c_{l,n} |  \| \tilde f \|_\infty  \frac{2}{n} E^*  \left[\left(\eps^*_{s_u+r-j} \eps^*_{s_u+r+h-l}\right)^2 + (Z^*_{t_2,n})^2\right]  \\
& \leq  \| \tilde f \|_\infty\frac{2}{n} (  \max(\hat \kappa_{4,n},\sigmah^2)  + E^* (Z^*_{t_2,n})^2)  \sum_{h\in \Z} |d(h)|  \sum_{l=0}^\infty |\hat c_{l,n} | \! \!  \!\sum_{j=r-M}^\infty \!\! \!|\hat c_{j,n}| \leq  \| \tilde f \|_\infty \frac{C}{n} \!\!\! \sum_{j=r-M}^\infty \!\!\!  |\hat c_{j,n}|.
\end{align*}
Consequently, the sequence $\{Z_{t,n}^*\}$ fulfills the weakly dependence condition of \citeasnoun{CLT1}, if  $2 C\sum_{j=r-M}^\infty |\hat  c_{j,n}| \leq \theta_r$  for some summable $(\theta_r)_{r \in \N}$. Since \linebreak$\sup_{\lambda \in (-\pi,\pi]} (\frac{\dx}{\dx \lambda})^3 \log \fn (\lambda) \leq C $ holds independently of $n$, it follows similarly to Lemma A.3  that $\sup_{j,n} | \hat c_{j,n} j^3| \leq C$. Hence, 
$2 C\sum_{j=r-M}^\infty |\hat c_{j,n}| \leq \sum_{j=r-M}^\infty {C}j^{-3}=: \theta_r \text{ for all }r > M \text{ and for some } C>0.$
If $r \leq M$ we set $\theta_r:=C$. Then it holds $\sum_{r=0}^\infty \theta_r = C (M +1 + \sum_{r=M+1}^\infty \sum_{j=r-M}^\infty j^{-3}) = C(M +1 + \sum_{j=1}^\infty j^{-2})< \infty.$
Regarding the variance of $T_{n}^{*M}$, consider firstly
$T_{n}^{*M}%= \frac{1}{\sqrt n} \sum_{t=1}^n \sum_{h=1-t}^{M \land (n-t)} d(h) \Xs t \Xs{t+h}
= {1}/{\sqrt n} \sum_{h=-n+1}^M \sum_{t=1 \lor (1-h)}^{n \land (n-h)} d(h) \Xs t \Xs{t+h}.$
Using the linear process structure of $\{\Xs t\}$ and additionally, let $c_{j}=0$ for all $j < 0$,  we get
\begin{align*}
\var(&T_{n}^{*M})=\! \! \! \!\!\!\!\! \sum_{h_1,h_2=-n+1}^M \! \!\!\! 1/n \!\!\!\! \sum_{t=1 \lor (1-h_1)}^{n \land (n-h_1)}\!\!\!\!\!d(h_1)\!\!\!\! \sum_{s=1 \lor (1-h_2)}^{n \land (n-h_2)}\! \!\!\!\!\! \!\! d(h_2) \left( E^*(\Xs t \Xs{t+h_1} \Xs s \Xs{s+h_2}) -\gammah(h_1)\gammah(h_2) \right) \\
%&=\! \! \! \!\!\!\!\! \sum_{h_1,h_2=-n+1}^M \! \!\!\! 1/n \!\!\!\!\sum_{t=1 \lor (1-h_1)}^{n \land (n-h_1)} \!\!\!\!\!d(h_1)\!\!\!\!\sum_{s=1 \lor (1-h_2)}^{n \land (n-h_2)} \!\! d(h_2) \Big( \sum_{j=0}^\infty \hat c_{j,n}\hat  c_{j+h_1,n}\hat  c_{j+s-t,n}\hat  c_{j+s+h_2-t,n} (\hat \kappa_{4,n}-3)  \\
%& \qquad \qquad + \gammah(s-t)\gammah(s-t+h_2-h_1)+\gammah(s+h_2-t)\gammah(s-t-h_1)\Big) \\
%&=\frac{1}{n} \sum_{h_1=1}^M \sum_{h_2=1}^M  \sum_{k \in \Z} d(h_1) d(h_2) \Big( \sum_{j=0}^\infty \hat c_{j,n} \hat c_{j+h_1,n}\hat  c_{j+k,n} \hat c_{j+k+h_2,n} (\hat \kappa_{4,n}-3)  \\
%& + \gammah(k)\gammah(k+h_2-h_1)+\gammah(k+h_2)\gammah(k-h_1)\Big) \sum_{t=1 \lor (1-h_1)}^{n \land (n-h_1)} \sum_{s=1 \lor (1-h_2)}^{n \land (n-h_2)} \!\!\!\!\ind_{\{k=s-t\}} \\
&=\sum_{h_1,h_2=-n+1}^M  \sum_{k=-(n-1)}^{n-1} d(h_1) d(h_2) \Big( \sum_{j=0}^\infty \hat  c_{j,n} \hat c_{j+h_1,n} \hat c_{j+k,n}\hat  c_{j+k+h_2,n} (\hat \kappa_{4,n}-3)  \\
& \qquad +\gammah(k) \gammah(k+h_1-h_2)+\gammah(k+h_2)\gammah(k-h_1)\Big)  \frac{\max(0,n-(|k|+|h_1-h_2|))}{n}. %\\
\end{align*}
Since $\hat \kappa_{4,n}$ is a consistent estimator of $\kappa_4$ and $\sum_{j=0}^\infty | c_j - \hat  c_{j,n}| = o_P(1)$  from which it follows that $\sum_{k \in \Z} \gammah(k) \gammah(k+x)= \sum_{k\in \Z} \gamma(k) \gamma(k+x) + o_P(1)$, we have that the last term is equal to
\begin{align*}
&=\sum_{h_1=-n+1}^M \sum_{h_2=-n+1}^M  \sum_{k=-(n-1)}^{n-1} d(h_1) d(h_2) \Big( \sum_{j=0}^\infty c_{j} c_{j+h_1} c_{j+k} c_{j+k+h_2} ( \kappa_{4}-3)  \\
& + \gamma(k)\gamma(k+h_1-h_2)+\gamma(k+h_2)\gamma(k-h_1)\Big)  \frac{\max(0,n-(|k|+|h_1-h_2|))}{n} + o_P(1).
\end{align*}
 We then have for the first term of the last equality above, that it equals, as $n \to \infty$, to 
%\begin{align*}
%&\sum_{h_1=-n+1}^M \sum_{h_2=-n+1}^M  d(h_1) d(h_2) \Big( \sum_{j=0}^\infty c_{j} c_{j+h_1}  \sum_{k\in \Z}c_{j+k+h_1-h_2} c_{j+k+h_1} ( \kappa_{4}-3)  \\
%& \qquad \qquad +  \sum_{k\in \Z}\gamma(k+h_1-h_2)\gamma(k)+\gamma(k+h_1)\gamma(k-h_2)\Big) \\
$\sum_{h_1,h_2=-n+1}^M d(h_1)  d(h_2) \Big(\gamma(h_1) \gamma(-h_2) ( \kappa_{4}/\sigma^4-3) + \sum_{k\in \Z}\gamma(k+h_1-h_2)\gamma(k)+\gamma(k+h_1)\gamma(k-h_2)\Big).$
%\end{align*}
Since $\sum_{j=0}^\infty | j c_j|< \infty$ and $\sum_{k \in \Z}  | k \gamma(k)| < \infty$, the second term of the same equality is of order  $\mathcal{O}(1/n)$. Hence, we have in probability, as $n \to \infty$, 
$
\var(T_{n}^{*M}) \to \!\sum_{h_1,h_2=-\infty}^M d(h_1)  d(h_2) \Big(\gamma(h_1) \gamma(-h_2) ( \kappa_{4}/\sigma^4-3) +\! \sum_{k\in \Z}\gamma(k+h_1-h_2)\gamma(k)+\gamma(k+h_1)\gamma(k-h_2)\Big).
$
Using the strictly stationarity of the process $(\Xs t)$ and  since $\sum_{k \in \Z} | d(h)| < \infty$ and $E(\Xs t)^4 < \infty$, we can verify Lindberg's condition by means of Lebesgue's dominated convergence theorem, that is
%\begin{align*}
$
\bigg|  {1}/{n}\sum_{t=1}^n   \sum_{h_1,h_2=1-t}^{M \land (n-t)}  d(h_1) d(h_2) \\
\times  E^* \left( ( {\Xs t}^2 \Xs{t + h_1} \Xs{t +h_2} - \gammah(h_1) \gammah(h_2) ) \ind_{\{|  \sum_{h_3=1-t}^{M \land (n-t)} d(h_3) (\Xs t \Xs{t+h_3} - \gammah(h_3)) | >\sqrt n \eps\}} \right) \bigg| \quad \,\\
 \leq  %\frac{1}{n}\sum_{t=1}^n     \sum_{h_1,h_2 \in \Z}|d(h_1)d(h_2)|E^* \left( ( |(\Xs t)^2 \Xs{t + h_1} \Xs{t +h_2} | + |\gammah(h_1) \gammah(h_2)| ) \ind_{\{\sum_{h_3 \in \Z } |d(h_3)| ( |\Xs t \Xs{t+h_3}| + |\gammah(h_3)|) >  \sqrt n \eps\}} \right)  \\ &= 
 \sum_{h_1,h_2 \in \Z} |d(h_1)d(h_2)| \\
\times E^* \left( ( |(\Xs 1)^2 \Xs{1 + h_1} \Xs{1 +h_2} | + |\gammah(h_1) \gammah(h_2)| ) \ind_{\{\sum_{h_3 \in \Z} |d(h_3)| ( |\Xs 1 \Xs{1+h_3}| + |\gammah(h_3)|) > \sqrt n \eps\}} \right) \to 0,$\\$  \text{as } n \to \infty.$
%\end{align*}
Therefore, by the central limit theorem for triangular arrays of weakly dependent random variables given in \citeasnoun{CLT1}, we have that, as $n\to \infty$,
$
\sqrt{n} ( T_{n}^{*M} - E^* T_{n}^{*M}) \to \mathcal{N} (0,\sigma_M^2), n \to \infty , \text{ in probability}.
$
Now, using a version in probability of Theorem 4.2 of \citeasnoun{billingsley} the proof of the theorem is concluded, since additionally to the converges for any fixed $M$, we have
$
\lim_{M \to \infty} \sigma_M^2 = \\ \sum_{h_1=-\infty}^\infty \sum_{h_2=-\infty}^\infty  d(h_1) d(h_2) \Big(\gamma(h_1) \gamma(-h_2) ( \kappa_{4}/\sigma^4-3) + \sum_{k\in \Z}\gamma(k+h_1-h_2)\gamma(k)+\gamma(k+h_1)\gamma(k-h_2)\Big).
$
Finally,  condition (3) of Theorem 4.2 of \citeasnoun{billingsley}  holds in probability since by Tchebysheff's inequality we have \\
%\begin{align*}
$
%$
%\lim_{M \to \infty} & \limsup_{n\to \infty} P\left( \left| \frac{1}{\sqrt n } \sum_{t=1}^n \sum_{h=1-t}^{M \land (n-t)} d(h) \left( \Xs t \Xs{t+h} - \gammah(h)\right) - \frac{1}{\sqrt n } \sum_{t=1}^n \sum_{h=1-t}^{(n-t)} d(h) \left( \Xs t \Xs{t+h} - \gammah(h)\right) \right| > \eps \right) \\
\lim_{M \to \infty} \limsup_{n\to \infty} P\left( \left|  \sum_{t=1}^n \sum_{h=M+1}^{n-t} d(h) \left( \Xs t \Xs{t+h} - \gammah(h)\right)  \right| > \sqrt n\eps \right) \\ \leq%\lim_{M \to \infty}  \limsup_{n\to \infty} \var\left(\frac{1}{\sqrt n } \sum_{t=1}^n \sum_{h=M+1}^{n-t} d(h) \left( \Xs t \Xs{t+h} \right)\right) / (\eps^2) \\&=
\lim_{M \to \infty}  \limsup_{n\to \infty} \var\left( \sum_{h=M+1}^n  d(h) \sum_{t=1}^{n-h} \left( \Xs t \Xs{t+h} \right)\right) / (n\eps^2) ,
$
%\end{align*}
%$
together with similar calculations as in the evaluation of $\var(T_{n}^{*M})$ and $\sum{k \in \Z}| \gamma(k)|< \infty$ we get 
%\begin{align*}
$
\lim_{M \to \infty} \sum_{h_1,h_2=M+1}^\infty d(h_1) d(h_2) O_P(1)= 0,
%\end{align*}
$
by the fact that $\sum_{h \in \Z} | d(h)| < \infty$. Using this summability property and the same arguments as in the calculation of $\var(T_n^{*M})$ leads to, as $n\to \infty$ and in probability, $\var (T_n^*)\to\sum_{h_1=-\infty}^\infty \sum_{h_2=-\infty}^\infty  d(h_1) d(h_2)$\linebreak$ \Big(\gamma(h_1) \gamma(-h_2) ( \kappa_{4}/\sigma^4-3) + \sum_{k\in \Z}\gamma(k+h_1-h_2)\gamma(k)+\gamma(k+h_1)\gamma(k-h_2)\Big)$ which can be written in frequency domain as \\ $(\kappa_{4}/\sigma^4-3) (\int_{0}^{2\pi} f(\lambda) \sum_{h \in \Z} d(h) \exp(ih\lambda)\dlambdaB)^2 + 4 \pi \int_{0}^{2\pi} |f(\lambda)  \sum_{h \in \Z} d(h) \exp(ih\lambda)|^2 \dlambdaB$. \linebreak Since the statistic in Definition \ref{statistic} has for linear processes the same asymptotic distribution as the one derived above, the assertion of the theorem follows by the triangular inequality.
\end{proof}

\begin{proof}[Proof of Theorem \ref{validity2}]
Without loss of generality, let $X_t^*$ be centered. Since $\sum_{h \in \Z} | \gamma(h)| < \infty$, we have $\sup_{\lambda \in [-\pi,\pi]} f(\lambda) < \infty$. %Furthermore, due to $\sum_{k=0}^\infty |\ckn|< C$, we have $\sum_{h\in \Z} |\gammah(h)| < C$. 
To proof Theorem \ref{validity2} we use a central limit theorem for $M$-dependent sequences, see \citeasnoun{romano2000more} and a version in probability of Theorem 4.2 of \citeasnoun{billingsley}. \\%, see also \citeasnoun[Prop. 6.3.9]{BrockwellDavis1991}. \\
Fix $M \in \N$ and consider the $M$-dependent process $X_{t,n,M}^*=\sum_{k=0}^M \ckn \eps_{t-k}^*, t \in \Z$ with spectral density $\hat f_M(\lambda)=|\sum_{k=0}^M \ckn \exp(i\lambda k)| ^2 \sigmah^2/(2\pi)$ and autocovariance $\hat \gamma_{n,M} (h)=\int_{-\pi}^\pi \hat f_M(\lambda) \exp(i h \lambda) \dlambdaB$. Applying a central limit theorem for $M$-dependent sequences we show that $\sum_{t=1}^n n^{-1/2} X_{t,n,M}^* \overset{D}{\to} \mathcal{N} (0, 2\pi f_M(0)), \text {in probability,}$
where \\ $f_M(0)=|\sum_{k=0}^M c_k|^2 \sigma^2/(2\pi)$. Theorem \ref{consistence} gives for the variance, let $n \geq M$, \\
%\begin{align*}
$
n^{-1}\var\left(\sum_{t=1}^n  X_{t,n,M}^*\right)  = \sum_{h=-M}^{M} (1-|h|/n) \hat \gamma_M(h) =(2\pi) \hat f_M(0) + O_P(1/n) \to  2\pi f_M(0),
$
%\end{align*}
 in probability, as $n \to \infty$. If $f_M(0)=0$ the assertion follows. Assume that $f_M(0)>0$. Since $\{X_{t,n,M}^*\}$ is stationary, we have for $k\geq M$ and all $a \in \N$ $\var(\sum_{t=a}^{a+k-1} n^{-1/2} X_{t,n,M}^*)=1/n \var(\sum_{t=1}^{k}   X_{t,n,M}^*)= 2\pi \hat f_M(0) k/n + O_P(1/n)$. Furthermore, the process  $\{X_{t,n,M}^*\}$ has i.i.d. innovations with a finite fourth moment. Thus, the fourth moment can easily be bounded; we have
$
E \left( n^{-1/2} X_{1,n,M}^* \right)^4 = \left(E( (\eps_1^*)^4 -3 \sigmah^4) \sum_{k=0}^M \ckn^4 + 3 \hat \gamma_M(0)^2\right)/n^2 =\mathcal O_P(1/n^2).
$
Consequently, the conditions of Theorem 2.1 of \citeasnoun{romano2000more} can be easily verified and it follows that $\sum_{t=1}^n n^{-1/2} X_{t,n,M}^* \overset{D}{\to} \mathcal{N} (0, 2\pi f_M(0)),$ in probability. Furthermore, the absolute summability of ${\gamma(h)}$ implies, as $M\to \infty$, $f_M(0)\to f(0)$. Since $\{\ckn\}$ is absolutely summable, it can be shown that the $M$-approximation used is sufficiently close. The absolute summability of  $\{\ckn\}$ implies absolute summability of $\{\gammah(h)=E^* X_{t+h,n}^* X_{t,n}^* \}$. Let $\delta>0$,
%\begin{align*}
$
\lim_{M\to \infty}  \limsup_{n\to \infty}  \, P( |\sum_{t=1}^n X_{t,n,M}^*  -  \sum_{t=1}^n X_{t,n}^* | > \sqrt n\delta) \delta^2%=\lim_{M\to \infty} \limsup_{n\to \infty} \, P\left( |\sum_{t=1}^n \frac{1}{\sqrt n} \sum_{k=M+1}^\infty \ckn \eps_{t-k}^* | > \delta\right)\\
\leq %\lim_{M\to \infty} \limsup_{n\to \infty} \sum_{h=-n+1}^{n-1} \sum_{k=M+1}^\infty (1-|h|/n) \ckn \hat c_{k+h,n}\sigmah^2 \\
%& =
 \lim_{M\to \infty} \limsup_{n\to \infty} \sum_{h=M+1}^\infty \gammah(h) - \sum_{h=-n+1}^{n-1} |h|/n \linebreak   \sum_{k=M+1}^\infty  \ckn \hat c_{k+h,n} \sigmah^2-  \sum_{|h|\geq n} \sum_{k=M+1}^\infty \ckn \hat c_{k+h,n} \sigmah^2 
=\lim_{M\to \infty} \sum_{h=M+1}^\infty \gamma(h) = 0,
$
%\end{align*}
%\begin{align*}
%&\lim_{M\to \infty}  \limsup_{n\to \infty}  \, P\left( |\sum_{t=1}^n \frac{1}{\sqrt n} X_{t,n,M}^*  -  \sum_{t=1}^n \frac{1}{\sqrt n} X_{t,n}^* | > \delta\right) \delta^2%=\lim_{M\to \infty} \limsup_{n\to \infty} \, P\left( |\sum_{t=1}^n \frac{1}{\sqrt n} \sum_{k=M+1}^\infty \ckn \eps_{t-k}^* | > \delta\right)\\
%\\&\leq %\lim_{M\to \infty} \limsup_{n\to \infty} \sum_{h=-n+1}^{n-1} \sum_{k=M+1}^\infty (1-|h|/n) \ckn \hat c_{k+h,n}\sigmah^2 \\
%%& =
% \lim_{M\to \infty} \limsup_{n\to \infty} \sum_{h=M+1}^\infty \!\!\!\gammah(h) -\!\!\! \sum_{h=-n+1}^{n-1} \!\!\!|h|/n \!\! \!\!\sum_{k=M+1}^\infty \!\!\! \ckn \hat c_{k+h,n} \sigmah^2- \sum_{|h|\geq n} \sum_{k=M+1}^\infty \ckn \hat c_{k+h,n} \sigmah^2 \\
%&=\lim_{M\to \infty} \sum_{h=M+1}^\infty \gamma(h) = 0,
%\end{align*}
where the last equation follows by the absolute summability of $\gamma(h)$. Thus,  the assertion follows with a version in probability of Theorem 4.2 of \citeasnoun{billingsley}. \\
Since $\fn$ is a uniformly consistent estimator and $\{\gammah\}$ is absolutely summable, we have 
%\begin{align*}
$
\var \left(\frac{1}{\sqrt n}\sum_{t=1}^n  X_{t,n}^*\right) = \sum_{h=-n+1}^{n-1} \left( 1- {|h|}/{n}\right) \gammah(h) = \fn(\lambda) - \sum_{|h|\geq n} \gammah(h) - \\\sum_{h=-n+1}^{n-1} {|h|}/{n} \gammah(h) = f(\lambda) + o_P(1)- \sum_{|h|\geq n} \gammah(h) - \sum_{h=-n+1}^{n-1} {|h|}/{n} \gammah(h) {\to}  f(\lambda),
$
%\end{align*}
as $n\to \infty$, in probability. The assertion follows by the triangular inequality.
\end{proof}
\begin{proof}[Proof of Corollary \ref{corollary0}]
By Applying the SDDB to the time series   \linebreak$Y_1,\dots, Y_n$, validity of the SDDB for this statistic can be derived using the delta-method and similar arguments as those used the proof of Theorem \ref{validity2}.
\end{proof}
\begin{proof}[Corollary \ref{corollary1}]
Since $\tilde f_n(0)$ is a consistent estimator, we have with Slutsky's Theorem $\sqrt n ( \Xn-\mu)/(2 \pi \tilde f_n(0))^{1/2} \to \mathcal{N} (0,1)$. Furthermore, we have for $0 < \eps_n < f(0),\, \eps_n \to 0,$ as $n \to \infty,$ that $E (n (\Xn-\mu)^2)/( 2\pi \tilde f_n(0))\to 1$, since
\begin{align*}
 E\left[\frac{n (\Xn-\mu)^2}{ 2\pi (f(0)+\tilde f_n(0)-f(0)} \ind_{\{| \tilde f_n(0)-f(0)| < \eps_n\}} + \frac{n (\Xn-\mu)^2}{( 2\pi \tilde f_n(0)}\ind_{\{|\tilde f_n(0) - f(0)| \geq \eps_n\}}\right]\\
\leq  E \frac{n (\Xn-\mu)^2}{ 2\pi f(0) + \eps_n} + E \frac{n (\Xn-\mu)^2}{ 2\pi \delta} E\ind_{\{|\tilde f_n(0) - f(0)| \geq \eps_n\}} \to 1, 
\end{align*}
and $
E {n (\Xn-\mu)^2}/\{ 2\pi \tilde f_n(0)\} \geq E {n (\Xn-\mu)^2}/\{ 2\pi \tilde f(0)+\eps_n\}  \ind_{\{| \tilde f_n(0)-f(0)| < \eps_n\}} \to 1,
$
 as $n \to \infty.$ \\
For a valid bootstrap approximation it is necessary that $\tilde f_n^*$ is a consistent estimator. The differentiability of $\fn$ ensures that the corresponding autocovariance function fulfills $|\gammah(h) |\leq |h| ^{-2 + \eps} C$ for some $\eps>0$ and for all $n \in \N$. Since $\{X_t^*, t\in \Z\}$ possesses a one-sided MA representation with i.i.d. innovations, we have with the absolute summability of $\{\ckn\}$ and the boundedness of fourth moment of the innovations that $\sup_t \sum_{h_1,h_2,h_3 \in \Z} \cum^*(X_t^*,X_{t+h_1}^*,X_{t+h_2}^*,X_{t+h_3}^*) < C$. Thus, similarly to Lemma A.2 in \cite{Jentsch2015} it can be shown that $\sup_\lambda | \fn^*(\lambda) -f (\lambda) | {\to} 0,$ in probability, where $\fn^*$ is a lag window SD estimator based on the pseudo observations $X_1^*,\dots,X_n^*$ and fulfilling Assumption 2.1 of \cite{Jentsch2015}. We construct our consistent estimator as $\tilde f_n^*(0) = \fn(0)  + \delta \ind_{\{\fn(0) < \delta\}}$. Using Theorem \ref{validity2} and the same arguments as above we get $\sqrt n (\Xn^* -\Xn)/(2 \pi \tilde f_n^*(0))^{1/2} \to \mathcal {N} (0,1)$, as $n\to \infty$, in probability. Furthermore, we have $E^* n (\Xn^* -\Xn)^2/(2 \pi \tilde f_n^*(0))=1$ and the assertions follows by the triangular inequality.
\end{proof}
\begin{proof}[Corollary \ref{corollary2}]
The assumption  $\sum_{k \in \N} |k \hat c_{k,n}|\leq C$ of Theorem \ref{validity1} ensures that $\sum_{k \in \N} |k \hat c_{k,n}|^2\leq \tilde C$. Furthermore, we have $E^* (\eps_t^*)^8< \infty$ independently from $n$. Thus the nonparametric estimator $\tilde \kappa_4^*$ of \citet{Fragkeskou2015} for the fourth moment of the innovation $\{\eps^*_t\}$ can be applied and is consistent under this conditions. The assumption $\sum_{k \in \N} |k \hat c_{k,n}|\leq C$ ensures that the corresponding autocovariance function $|\gammah(h) |\leq |h| ^{-2 + \eps} C$ for some $\eps>0$ and for all $n \in \N$. Thus, the consistency of a lag-window SD estimator given in \cite{Jentsch2015} follows by the same arguments as in the proof of Corollary \ref{corollary1}. Since $\tau^2$ is a continuous transformation of $\kappa_4$ and $f$, we construct the following consistent estimators of $\tau^2>\delta$ where $0 < \eps_n< \delta$ and $\eps_n \to 0,$ as $n \to \infty$, $\tilde \tau_2^*=\max(\eps_n,(\tilde \kappa_{4}^*/\sigma^4-3) (\int_{0}^{2\pi} \tilde f_n^*(\lambda) \sum_{h \in \Z} d(h) \exp(ih\lambda)\dlambdaB)^2 + 4 \pi \int_{0}^{2\pi} |\tilde f_n^*(\lambda)  \sum_{h \in \Z} d(h) \exp(ih\lambda)|^2 \dlambdaB)$ which is based on $X_1*,\dots,X_n^*$ and 
$\tilde \tau_2=\max(\eps_n,(\tilde \kappa_{4}/\sigma^4-3) (\int_{0}^{2\pi} \tilde f_n(\lambda) \sum_{h \in \Z} d(h) \exp(ih\lambda)\dlambdaB)^2 + 4 \pi \int_{0}^{2\pi} |\tilde f_n(\lambda)  \sum_{h \in \Z} d(h) \exp(ih\lambda)|^2 \dlambdaB)$ which is based on $X_1,\dots,X_n$. 
The assertions follows with Theorem \ref{validity1} and the same arguments as in the proof of Corollary \ref{corollary1}.
\end{proof}

\section{Acknowledgments}
We are very grateful for the extremely helpful suggestions of two anonymous reviewers and an Associate Editor. Their careful reading of a previous version of the paper and their thorough reports led to a considerable improvement of the present paper.

\bibliographystyle{rss}

\bibliography{bib}

\end{document}